\documentclass[12pt, reqno]{amsart}
\usepackage{amsmath, amsthm, amscd, amsfonts, amssymb, graphicx, color}
\usepackage[bookmarksnumbered, colorlinks, plainpages]{hyperref}
\hypersetup{colorlinks=true,linkcolor=red, anchorcolor=green, citecolor=cyan, urlcolor=red, filecolor=magenta, pdftoolbar=true}

\newtheorem{theorem}{Theorem}[section]
\newtheorem{lemma}[theorem]{Lemma}
\newtheorem{proposition}[theorem]{Proposition}
\newtheorem{corollary}[theorem]{Corollary}
\theoremstyle{definition}
\newtheorem{definition}[theorem]{Definition}
\newtheorem{example}[theorem]{Example}

\newtheorem{problem}[theorem]{Problem}
\theoremstyle{remark}
\newtheorem{remark}[theorem]{Remark}
\numberwithin{equation}{section}

\begin{document}

\setcounter{page}{1}

\title[$up$-convergence \MakeLowercase{in} 
LNVL\MakeLowercase{s}]{Unbounded $p$-convergence in lattice-normed vector lattices}

\author[A. Ayd{\i}n, E. Emelyanov, N. Erkur\c{s}un \"Ozcan, \MakeLowercase{and} M. Marabeh]{Abdullah Ayd{\i}n$^{1}$, Eduard Emelyanov$^2$, Nazife Erkur\c{s}un \"Ozcan$^3$, \MakeLowercase{and} Mohammad Marabeh$^4$$^{*}$}

\address{$^{1}$Department of Mathematics, Mu\c{s} Alparslan University, Mu\c{s}, 49250 , Turkey.}
\email{\textcolor[rgb]{0.00,0.00,0.84}{a.aydin@alparslan.edu.tr}}

\address{$^{2}$Department of Mathematics, Middle East Technical University, Ankara, 06800 Turkey.}

\email{\textcolor[rgb]{0.00,0.00,0.84}{eduard@metu.edu.tr}}

\address{$^{3}$Department of Mathematics,  Hacettepe University, Ankara, 06800, Turkey.}
\email{\textcolor[rgb]{0.00,0.00,0.84}{erkursun.ozcan@hacettepe.edu.tr}}

\address{$^{4}$Department of Applied Mathematics, Palestine Technical University-Kadoorie, P.O. Box Tulkarem - Jaffa Street 7, Palestine.}
\email{\textcolor[rgb]{0.00,0.00,0.84}{m.maraabeh@gmail.com}}


\let\thefootnote\relax\footnote{Copyright 2016 by the Tusi Mathematical Research Group.}



\date{Received: xxxxxx; Revised: yyyyyy; Accepted: zzzzzz.
\newline \indent $^{*}$Corresponding author}

\begin{abstract}
	
A net $x_\alpha$ in a lattice-normed vector lattice $(X,p,E)$ is unbounded $p$-convergent
to $x\in X$ if $p(|x_\alpha-x|\wedge u)\xrightarrow{o} 0$ for every $u\in X_+$. This convergence 
has been investigated recently for $(X,p,E)=(X,\lvert\cdot \rvert,X)$ under the name of 
$uo$-convergence, for $(X,p,E)=(X,\lVert\cdot\rVert,{\mathbb R})$ under the name of $un$-convergence, and also for $(X,p,{\mathbb R}^{X^*})$, where $p(x)[f]:=|f|(|x|)$, under the name $uaw$-convergence. In this paper we study general properties of the unbounded $p$-convergence.	
	


\end{abstract} \maketitle

\section{Introduction and preliminaries}
Lattice-valued norms on vector lattices provide natural and efficient tools in the theory of vector lattices. 
It is enough to mention the theory of lattice-normed vector lattices (see, for example, \cite{KK,K,E}). The main aim of the present paper is to illustrate 
usefulness of lattice-valued norms for investigation of different types of {\em unbounded convergences} in vector lattices, 
which attracted attention of several authors in series of recent papers \cite{GX,G,GTX,GLX,EM,DOT,KMT,AEEM2,Z,DEM-a,Tay1,DEM-b,EV,Tay2,KT,EEG}. 

The {\em $uo$-convergence} was introduced in \cite{N} under the name {\em individual convergence}, and the {\em $un$-convergence} 
was introduced in \cite{T} under the name {\em d-convergence}. We refer the reader for an exposition on $uo$-convergence to \cite{GTX,GX} 
and on $un$-convergence to \cite{DOT} (see also recent paper \cite{KMT}). For applications of $uo$-convergence, we refer to \cite{EM,GTX,GX,GLX,M}. 
Throughout the paper, all vector lattices are assumed to be real and Archime\-dean.

Recall that a net $(x_\alpha)_{\alpha\in A}$ in a vector lattice $X$ is order convergent (or $o$-convergent, for short) to $x\in X$, if there exists another net 
$(y_\beta)_{\beta\in B}$ satisfying $y_\beta \downarrow 0$, and for any $\beta\in B$, there exists $\alpha_\beta\in A$ such that $|x_\alpha-x|\leq y_\beta$ 
for all $\alpha\geq\alpha_\beta$. In this case we write $x_\alpha\xrightarrow{o} x$. In a vector lattice $X$, a net $x_\alpha$ is unbounded order convergent 
(or $uo$-convergent, for short) to $x\in X$ if $|x_\alpha-x|\wedge u\xrightarrow{o}0$ for every $u\in X_+$. In this case we write $x_\alpha\xrightarrow{uo}x$. 
The $uo$-convergence is an abstraction of a.e.-convergence in $L_p$-spaces for $1\leq p<\infty$, \cite{GTX, GX}. In a normed lattice $(X,\left\|\cdot\right\|)$, 
a net $x_\alpha$ is unbounded norm convergent to $x\in X$, written as $x_\alpha\xrightarrow{un}x$, if $\left\||x_\alpha-x|\wedge u\right\|\to 0$ for every $u\in X_+$. 
Clearly, if the norm is order continuous, then $uo$-convergence implies $un$-convergence. 
For a finite measure $\mu$, $un$-convergence of sequences in $L_p(\mu)$,\ $1\leq p<\infty$, is equivalent to convergence in measure, see \cite{DOT,T}. 
Recently, Zabeti \cite{Z} introduced the following notion. A net $x_\alpha$ in a Banach lattice $X$ is said 
to be {\em unbounded absolute weak convergent} (or {\em $uaw$-convergent}, for short) to $x\in X$ if, for each $u\in X_+$, $|x_\alpha-x|\wedge u\to 0$ weakly.  

Let $X$ be a vector space, $E$ be a vector lattice, and $p:X \to E_+$ be a vector norm (i.e. $p(x)=0\Leftrightarrow x=0$, 
$p(\lambda x)=|\lambda|p(x)$ for all $\lambda\in\mathbb{R}$, $x\in X$, and $p(x+y)\leq p(x)+p(y)$ for all $x,y\in X$), then the triple $(X,p,E)$ 
is called a {\em lattice-normed space}, abbreviated as LNS. We say that elements $x$ and $y$ of an LNS $X$ are $p$-disjoint if their lattice norms are disjoint, and 
abbreviate this by $x{\perp_{\mathbf p}} y$. The lattice norm $p$ in an LNS $(X,p,E)$ is said to be {\em decomposable} if, for all $x\in X$ and
$e_1,e_2\in E_+$, from $p(x)=e_1+e_2$ it follows that there exist $x_1,x_2\in X$ such that $x=x_1+x_2$ and $p(x_k)=e_k$ for $k=1,2$. 
We abbreviate the convergence $p(x_{\alpha}-x)\xrightarrow{o}0$ as $x_\alpha\xrightarrow{p}x$ and say in this case that $x_\alpha$ $p$-converges to $x$.  
We refer the reader for more information on LNSs to \cite{KK,K}. 

If, in addition, $X$ is a vector lattice and the vector norm $p$ is monotone (i.e. $|x|\leq |y|\Rightarrow p(x)\leq p(y)$), then the 
triple $(X,p,E)$ is called {\em lattice-normed vector lattice}, abbreviated as LNVL. In an LNVL $(X,p,E)$, $p$-disjointness implies disjointness.
Indeed, let $x{\perp_{\mathbf p}} y$ and $0\le u\le|x|\wedge|y|$. Then $p(u)\le p(|x|\wedge |y|)\le p(x)\wedge p(y)=0$ and hence $u=0$. Thus $x\bot y$. 
We shall make difference between two notions of bands in an LNVL $X=(X,p,E)$. More precisely, a subset $B$ of $X$ is called a {\em band} if it is a band in 
the usual sense of the vector lattice $X$. Following to \cite[2.1.2.]{K}, we say that a subset $B$ of $X$ is a {\em $p$-band} if 
$$
B=M^{{\perp_{\mathbf p}}}=\{x\in X:(\forall m\in M) \ x{\perp_{\mathbf p}} m\}
$$
for some nonempty $M\subseteq X$. In general, there are many bands which are not $p$-bands. 
To see this, consider the normed lattice $({\mathbb R}^2,\|\cdot\|,{\mathbb R})$. It has four bands, but only two of them are $p$-bands. 
It is easy to see that any $p$-band is an order ideal. The following example shows that a $p$-band may not be a band in general.

\begin{example}\label{ExLNVL_1}
	Consider the LNVL $(c,p,c)$ with 
	$$
	p(x):=|x|+(\lim_{n\to\infty}|x_n|)\cdot \textbf{1}  \ \ \ \ (x=(x_n)_n \in c), 
	$$
	where $\textbf{1}$ denotes the sequence identically equal to 1.
	Take $M=\{e_1\}$. Then the $p$-band $M^{{\perp_{\mathbf p}}}=\{x\in c_0: x_1=0\}$ is not a band.
\end{example}

In Proposition \ref{bands}, we show that, under some mild conditions, every $p$-band is a band. 
Unless otherwise stated, we do not assume lattice norms to be decomposable. 
While dealing with LNVLs, we shall keep in mind also the following examples.

\begin{example}\label{ExLNVL_2}
	Let $X$ be a normed lattice with a norm $\lVert\cdot\rVert$. Then $X$ is the LNVL $(X,\lVert\cdot\rVert,{\mathbb R})$. 
\end{example}

\begin{example}\label{ExLNVL_3}
	Let $X$ be a vector lattice. Then $X$ is the LNVL $(X,\lvert\cdot\rvert,X)$. 
\end{example}

\begin{example}\label{ExLNVL_4}
	Let $X=(X,\lVert\cdot\rVert)$ be a normed lattice. Consider the closed unit ball $B_{X^*}$ of the dual Banach lattice $X^*$.
	Let $E=\ell^{\infty}(B_{X^*})$ be the vector lattice of all bounded real-valued functions on $B_{X^*}$. Define an $E$-valued norm $p$ 
	on $X$ by
	$$
	p(x)[f]:=|f|(|x|) \ \ \ (f\in B_{X^*})
	$$
	for any $x\in X$. The Hahn-Banach theorem ensures that $p(x)=0$ iff $x=0$. All other properties of lattice norm are obvious for $p$. 
	Thus $(X,p,E)$ is an LNVL. Notice also that the lattice norm $p$ takes values in the space $C(B_{X^*})$ of all continuous functions 
	on the $w^*$-compact ball $B_{X^*}$ of $X^*$. Hence, instead of $(X,p,\ell^{\infty}(B_{X^*}))$, one may also consider the LNVL $(X,p,C(B_{X^*}))$. 
\end{example}

\begin{example}\label{ExLNVL_5}
	Let $X$ be a vector lattice, $X^\#$ be the algebraic dual of $X$, and $Y$ be a sublattice of $X^\#$ such that $\left\langle X,Y\right\rangle$ is a dual system. 
	Define $p:X\to {\mathbb R}^Y$ by $p(x)[f]:=|f|(|x|)$. Then $(X,p,{\mathbb R}^Y)$ is an LNVL.
\end{example}

The LNVLs in Examples \ref{ExLNVL_1}, \ref{ExLNVL_2}, and \ref{ExLNVL_3} have decomposable norms. 
It can be shown easily that in Examples \ref{ExLNVL_4} and \ref{ExLNVL_5} the lattice norms are decomposable iff $\dim(X)=1$. 

We refer the reader for further examples of LNSs to \cite{K}. It should be noticed that the theory of lattice-normed spaces 
is well developed in the case of decomposable lattice norms (cf. \cite{KK,K}).
In general, we do not assume lattice norms to be decomposable. 

The structure of the paper is as follows. In Section 2, we study several notions related to LNVLs in parallel to the theory of Banach lattices. In particular, an LNVL $(X,p,E)$ is said to be: {\em $op$-continuous} 
if $X\ni x_\alpha\xrightarrow{o}0$ implies $x_\alpha\xrightarrow{p}0$; a {\em $p$-KB-space} if, for any $0\leq x_\alpha\uparrow$ with $p(x_\alpha)\leq e\in E$, there exists $x\in X$ satisfying 
$x_\alpha\xrightarrow{p}x$. We give a characterization of $op$-continuity in Theorem \ref{op-contchar}, and study some properties of $p$-KB-spaces, e.g. 
in Proposition \ref{$o$-closed sublattice of $p$-KB} and in Proposition \ref{KB}. A vector $e\in X$ is called a {\em $p$-unit} if, for any $x\in X_+$, $p(x-ne\wedge x)\xrightarrow{o}0$. 
Any $p$-unit is a weak unit, whereas strong units are $p$-units. For a normed lattice $(X,\lVert\cdot\rVert)$, a vector in $X$ is a $p$-unit in $(X,\lVert\cdot\rVert,\mathbb R)$ iff it is a quasi-interior point of the normed lattice $(X,\lVert\cdot\rVert)$. 

In Section 3, some basic theory of {\em unbounded $p$-convergence} in LNVLs is developed in parallel to $uo$- and $un$-convergences. For example, it is enough to check out the $uo$-convergence at a weak unit, 
while the $un$-convergence needs to be checked only at a quasi-interior point. Similarly, in LNVLs, {\em $up$-convergence} needs to be examined at a $p$-unit by Theorem \ref{$up$-conv by $p$-unit}. 

In Section 4, we introduce and study {\em $up$-regular} sublattices. Majorizing sublattices and projection bands are examples of $up$-regular sublattices 
by Theorem \ref{$up$-regular}. Also some further investigation of up-regular sublattices is carried out in certain LNVLs in this section.

In the last section, we study properties of mixed-normed LNVLs in Proposition \ref{mixed}, in Theorem \ref{pKB-pKB}, and in Theorem \ref{up-complete}. 
We also prove that in a certain LNVL, the $up$-null nets are ``$p$-almost disjoint'' (see Theorem \ref{$p$-version of Theorem 3.2 in DOT}). Those results 
generalize correspondent results from \cite{GTX,DOT}.

We refer the reader for unexplained notions and terminology to \cite{AB,K,LZ1}. 

\section{$p$-Notions in lattice-normed vector lattices}

Most of notions and results of this preliminary section are direct analogies of well-known facts of the theory of normed lattices.
We include them for convenience of the reader. They are also of certain proper interest and some of them will be used in further sections. 
In the present section, we define and study certain necessary notions such as: $op$-continuity of LNVLs, $p$-KB-spaces, $p$-Fatou spaces, $p$-units, etc. 
In particular, we characterize the $op$-continuity, prove some properties of $p$-KB-spaces, discuss $p$-dense subsets, and study $p$-units in LNVLs. 

\subsection{$p$-Continuity of lattice operations in LNVLs.}
The lattice operations in an LNVL $X$ are $p$-continuous in the following sense.

\begin{lemma}\label{LO are $p$-continuous}
	Let $(x_\alpha)_{\alpha \in A}$ and $(y_\beta)_{\beta \in B}$ be two nets in an LNVL $(X,p,E)$. 
	If $x_\alpha\xrightarrow{p}x$ and $y_\beta\xrightarrow{p}y$, then $(x_\alpha\vee y_\beta)_{(\alpha,\beta)\in A\times B} \xrightarrow{p} x\vee y$.
	In particular, $x_\alpha\xrightarrow{p}x$ implies that $x_\alpha^-\xrightarrow{p} x^-$.
\end{lemma}

However this result seems to be well-known, we did not find appropriate references for it and therefore, its elementary proof is included for the reader's convenience.

\begin{proof}
	There exist two nets $(z_\gamma)_{\gamma\in\Gamma}$ and $(w_\lambda)_{\lambda\in\Lambda}$ in $E$ satisfying $z_\gamma\downarrow 0$ 
	and $w_\lambda\downarrow 0$, and for all $(\gamma,\lambda)\in\Gamma\times\Lambda$ there are $\alpha_\gamma\in A$ and $\beta_\lambda\in B$ 
	such that $p(x_\alpha-x)\leq z_\gamma$ and $p(y_\beta-y)\leq w_\lambda$ for all $(\alpha,\beta)\geq(\alpha_\gamma,\beta_\lambda)$. 
	It follows from the inequality $|a\vee b-a\vee c|\leq |b-c|$ that 	
	\begin{equation*}
	\begin{split}
	p(x_\alpha \vee y_\beta - x\vee y)&=p(|x_\alpha \vee y_\beta -x_\alpha \vee y+x_\alpha \vee y- x\vee y|)\\ 
	&\leq p(|x_\alpha \vee y_\beta -x_\alpha \vee y|)+p(|x_\alpha \vee y- x\vee y|)\\ 
	&\leq p(|y_\beta -y|)+p(|x_\alpha-x|)\leq w_\lambda+z_\gamma
	\end{split}
	\end{equation*}
	for all $\alpha\geq\alpha_\gamma$ and $\beta\geq\beta_\lambda$. 
	Since $(w_\lambda+z_\gamma)\downarrow 0$, then $p(x_\alpha\vee y_\beta-x\vee y)\xrightarrow{o}0$. 
\end{proof}

\begin{definition}
	Let $(X,p,E)$ be an LNVL and $Y\subseteq X$. Then $Y$ is called {\em $p$-closed} in $X$ if, for any net $y_\alpha$ in $Y$ that is $p$-convergent to $x\in X$, it holds that $x\in Y$. 
\end{definition}

\begin{remark}\label{p-closedness} \
	\begin{enumerate}
		\item
		Every band is $p$-closed. Indeed, given a band $B$ in an LNVL $(X,p,E)$. If $B\ni x_\alpha\xrightarrow{p}x$, then, by Lemma 
		\ref{LO are $p$-continuous}, $|x_\alpha|\wedge |y|\xrightarrow{p} |x|\wedge|y|$ for any $y\in B^\perp$. Since $|x_\alpha|\wedge|y|=0$ for all $\alpha$, 
		then $|x|\wedge|y| =0$, and so $x\in B^{\perp\perp}=B$. 
		
		\item 
		Every $p$-band is $p$-closed. Indeed, let $B=M^{{\perp_{\mathbf p}}}$ for some nonempty $M\subseteq X$ and 
		$B\ni x_\alpha\xrightarrow{p} x_0\in X$. Take any $m\in M$. It follows from
		$$
		p(x_0)\wedge p(m)\le(p(x_0-x_\alpha)+p(x_\alpha))\wedge p(m)\le 
		$$ 
		$$
		p(x_0-x_\alpha)\wedge p(m)+p(x_\alpha)\wedge p(m)=p(x_0-x_\alpha)\wedge p(m)\xrightarrow{o}0,
		$$
		that $p(x_0)\wedge p(m)=0$. Since $m\in M$ is arbitrary, then $x_0\in B$.
	\end{enumerate}
\end{remark}

The following well-known property is a direct consequence of Lemma \ref{LO are $p$-continuous}.

\begin{proposition}\label{pos. cone is $p$-closed}
	The positive cone $X_+$ in any LNVL $X$ is $p$-closed.
\end{proposition}

Proposition \ref{pos. cone is $p$-closed} implies the following well-known fact.

\begin{proposition}\label{$p$-sup-inf}
	Any monotone $p$-convergent net in an LNVL $o$-converges to its $p$-limit.
\end{proposition}

\begin{proof}
	It is enough to show that if $(X,p,E)\ni x_\alpha\uparrow$ and $x_\alpha\xrightarrow{p}x$, then $x_\alpha\uparrow x$.
	
	Fix arbitrary $\alpha$. Then $x_\beta-x_\alpha\in X_+$ for $\beta\ge\alpha$.
	By Proposition \ref{pos. cone is $p$-closed}, $x_\beta-x_\alpha\xrightarrow{p}x-x_\alpha\in X_+$.
	Therefore, $x\geq x_\alpha$ for any $\alpha$. Since $\alpha$ is arbitrary, then $x$ is an upper bound of $x_\alpha$.
	
	If $y\geq x_\alpha$ for all $\alpha$, then, again by Proposition \ref{pos. cone is $p$-closed},
	$y-x_\alpha\xrightarrow{p} y-x\in X_+$, or $y\ge x$. 
	Thus $x_\alpha \uparrow x$.
\end{proof}

\subsection{Several basic $p$-notions in LNVLs}
We continue with several basic notions in LNVLs, which are motivated by their analogies from vector lattice theory.

\begin{definition}\label{$p$-notions}
	Let $X=(X,p,E)$ be an LNVL. Then 
	
	$(i)$ \  a net $(x_\alpha)_{\alpha \in A}$ in $X$ is said to be {\em $p$-Cauchy} if the net $(x_\alpha-x_{\alpha'})_{(\alpha,\alpha') \in A\times A}$ $p$-converges to $0$;
	
	$(ii)$ \ $X$ is called {\em $p$-complete} if every $p$-Cauchy net in $X$ is $p$-convergent;
	
	$(iii)$ a subset $Y\subseteq X$ is called {\em $p$-bounded} if there exists $e\in E$ such that $p(y)\leq e$ for all $y\in Y$;
	
	$(iv)$ $X$ is called {\em $op$-continuous} if $x_\alpha\xrightarrow{o}0$ implies that $p(x_\alpha)\xrightarrow{o}0$;
	
	$(v)$ \ $X$ is called a {\em $p$-KB-space} if every $p$-bounded increasing net in $X_+$ is $p$-convergent;
	
	$(vi)$ $p$ is said to be {\em additive on $X_+$} if $p(x+y)=p(x)+p(y)$ for all $x,y\in X_+$.
	
\end{definition}

\begin{remark}\label{on $p$-notions} \ \            
	\begin{enumerate}
		\item $p$-convergence, a $p$-Cauchy net, $p$-completeness, and $p$-boundedness in LNVLs are also known as {\em $bo$-convergence}, {\em a $bo$-fundamental net}, {\em $bo$-completeness}, 
		and {\em norm-boundedness} respectively $($see, e.g. \cite[p.48]{K}$)$.
		
		\item Clearly, any LNVL $(X,\lvert\cdot\rvert,X)$ is $op$-continuous.		
		
		\item In Definition \ref{$p$-notions}$(v)$ we do not require $p$-completeness of $X$.
		
		\item It is easy to see that a $p$-KB-space $(X,\left\|\cdot\right\|,\mathbb{R})$ is always $p$-complete (see, e.g. \cite[Ex.95.4]{Za}). 
		Therefore, the notion of $p$-KB-space coincides with the notion of $KB$-space. 
		
		\item Clearly, an LNVL $X=(X,\lvert\cdot\rvert,X)$ is a $p$-KB-space iff $X$ is order complete.
		
		\item Notice that, for a $p$-KB-space $X=(X,p,E)$ the vector lattice $p(X)^{\perp\perp}$ need not to be order complete.
		To see this, take a KB-space $(X,\lVert\cdot\rVert)$ and $E=C[0,1]$. Then the LNVL $X=(X,p,E)$ with $p(x):=\|x\|\cdot { \textbf{1}_{[0,1]}}$
		is clearly a $p$-KB-space, yet $p(X)^{\perp\perp}=E$ is not order complete.
	\end{enumerate}
\end{remark}

\begin{lemma}\label{$op$-cont-0}
	For an LNVL $(X,p,E)$, the following statements are equivalent.
	
	$(i)$ $X$ is {\em $op$-continuous};
	
	$(ii)$ $x_\alpha\downarrow 0$ in X implies $p(x_\alpha)\downarrow 0$.	
\end{lemma}

\begin{proof} 
	The implication $(i)\Rightarrow (ii)$ is trivial.
	
	$(ii)\Rightarrow(i)$: 
	Let $x_\alpha\xrightarrow{o}0$, then there exists a net $z_\beta\downarrow 0$ in $X$ 
	such that, for any $\beta$ there exists $\alpha_\beta$ so that $|x_\alpha|\leq z_\beta$ for all $\alpha\geq\alpha_\beta$. 
	Hence $p(x_\alpha)\leq p(z_\beta)$ for all $\alpha\geq\alpha_\beta$. By $(ii)$, $p(z_\beta)\downarrow 0$. 
	Therefore, $p(x_\alpha)\xrightarrow{o}0$ or $x_\alpha\xrightarrow{p}0$.
\end{proof}

From this proposition, it follows that the $op$-continuity in LNVLs is equivalent to the order continuity in the sense of \cite[2.1.4, p.48]{K}.
In the case of a $p$-complete LNVL, we have further conditions for $op$-continuity.

\begin{theorem}\label{op-contchar}
	For a $p$-complete LNVL $(X,p,E)$, the following statements are equivalent:
	
	$(i)$ \ \ $X$ is {\em $op$-continuous};
	
	$(ii)$ \ if $0\leq x_\alpha\uparrow\leq x$ holds in X, then $x_\alpha$ is a $p$-Cauchy net;
	
	$(iii)$ $x_\alpha\downarrow 0$ in X implies $p(x_\alpha)\downarrow 0$.	
\end{theorem}

\begin{proof} 
	$(i)\Rightarrow(ii)$: Let $0\leq x_\alpha\uparrow\leq x$ in X. By \cite[Lm.12.8]{AB}, there exists a net $y_\beta$ in X such that 
	$(y_\beta-x_\alpha)_{\alpha,\beta}\downarrow 0$. So $p(y_\beta-x_\alpha)\xrightarrow{o}0$, and hence the net $x_\alpha$ is $p$-Cauchy.
	
	$(ii)\Rightarrow(iii)$: Assume that $x_\alpha\downarrow 0$ in X. Fix arbitrary $\alpha_0$, then, for $\alpha\geq\alpha_0$,  $x_\alpha\leq x_{\alpha_0}$, 
	and so $0\leq (x_{\alpha_0}-x_\alpha)_{\alpha\geq\alpha_0}\uparrow\leq x_{\alpha_0}$. By (ii), the net $(x_{\alpha_0}-x_\alpha)_{\alpha\geq\alpha_0}$ is $p$-Cauchy, i.e. 
	$p(x_{\alpha^{'}}-x_\alpha)\xrightarrow{o}0$ as $\alpha_0\le\alpha,\alpha^{'}\to\infty$.  
	Since $X$ is $p$-complete, then there exists $x\in X$ satisfying $p(x_\alpha-x)\xrightarrow{o}0$ as $\alpha_0\le\alpha\to\infty$. 
	By Proposition \ref{$p$-sup-inf}, $x_\alpha\downarrow x$ and hence $x=0$. As a result, $x_\alpha\xrightarrow{p}0$
	and the monotonicity of $p$ implies $p(x_\alpha)\downarrow 0$.
	
	$(iii)\Rightarrow(i)$: It is just the implication $(ii)\Rightarrow(i)$ of Lemma \ref{$op$-cont-0}.
\end{proof}

\begin{corollary}\label{op + p implies o}
	Let $(X,p,E)$ be an $op$-continuous and $p$-complete LNVL, then $X$ is order complete. 
\end{corollary}

\begin{proof}
	Assume $0\leq x_\alpha\uparrow\leq u$, then by Theorem \ref{op-contchar} (ii), $x_\alpha$ is a $p$-Cauchy net and
	since $X$ is $p$-complete, then there is $x$ such that $x_\alpha\xrightarrow{p}x$. It follows from Proposition \ref{$p$-sup-inf}
	that $x_\alpha\uparrow x$, and so $X$ is order complete.
\end{proof}

\begin{corollary}\label{p-KB is op}
	Any $p$-KB-space is $op$-continuous. 
\end{corollary}

\begin{proof} 
	Let $x_\alpha\downarrow 0$. Take any $\alpha_0$ and let $y_\alpha:=x_{\alpha_0}-x_\alpha$ for $\alpha\ge\alpha_0$.
	Clearly, $y_\alpha\uparrow\le x_{\alpha_0}$. Hence $p(y_\alpha)\uparrow\le p(x_{\alpha_0})$ for $\alpha\ge\alpha_0$.
	Since $X$ is a $p$-KB-space, there exists $y\in X$ such that $p(y_\alpha-y)\xrightarrow{o}0$.
	Since $y_\alpha\uparrow$ and $y_\alpha\xrightarrow{p}y$, Proposition \ref{$p$-sup-inf} ensures that 
	$$
	y=\sup\limits_{\alpha\ge\alpha_0}y_\alpha=\sup\limits_{\alpha\ge\alpha_0}(x_{\alpha_0}-x_\alpha )=x_{\alpha_0},
	$$
	and hence $y_\alpha=x_{\alpha_0}-x_\alpha\xrightarrow{p} x_{\alpha_0}$ or $x_\alpha\xrightarrow{p}0$.
	Again by Proposition \ref{$p$-sup-inf} we get $p(x_\alpha)\downarrow 0$.
	So by Lemma \ref{$op$-cont-0}, $X$ is $op$-continuous. 
\end{proof}

\begin{proposition}
	Any $p$-KB-space is order complete.
\end{proposition}

\begin{proof}
	Let $X$ be a $p$-KB-space and $0\leq x_\alpha\uparrow\leq z\in X$. Then $p(x_\alpha)\leq p(z)$. 
	Hence the net $x_\alpha$ is $p$-bounded and therefore, $x_\alpha\xrightarrow{p}x$ for some $x\in X$. 
	By Proposition \ref{$p$-sup-inf}, $x_\alpha\uparrow x$.
\end{proof}

The following question arises naturally. 

\begin{problem}\label{Q1}
	Let $(X,p,E)$ be a $p$-KB-space. Is $(X,p,E)$ $p$-complete?
\end{problem}

We do not know the answer to Problem \ref{Q1} even under the assumption that $E$ is order complete. 

\begin{proposition}\label{$o$-closed sublattice of $p$-KB}
	Let $(X,p,E)$ be a $p$-KB-space, and $Y\subseteq X$ be an order closed sublattice. Then $(Y,p,E)$ is also a $p$-KB-space.
\end{proposition}

\begin{proof}
	Let $Y_+\ni y_\alpha\uparrow$ and $p(y_\alpha)\leq e\in E_+$ for all $\alpha$. Since $X$ is a $p$-KB-space, there exists $x\in X_+$ such that $y_\alpha \xrightarrow[]{p} x$. 
	By Proposition \ref{$p$-sup-inf}, we have $y_\alpha\uparrow x$, and so $x\in Y$, because $Y$ is order closed. Thus $(Y,p,E)$ is a $p$-KB-space.
\end{proof} 

It is clear from the proof of Proposition \ref{$o$-closed sublattice of $p$-KB}, that every $p$-closed sublattice $Y$ of a $p$-KB-space $X$ is also a $p$-KB-space.

\begin{proposition}\label{KB}
	Let $(X,p,E)$ be a $p$-complete LNVL, $E$ be atomic, and $p$ be additive on $X_+$. Then $X$ is a $p$-KB-space. 
\end{proposition}

\begin{proof}
	Let a net $x_\alpha$ in $X_+$ be increasing and $p$-bounded by $e\in E_+$. 
	If the net $x_\alpha$ is not $p$-Cauchy, then there exists an atom $a\in E$ such that 
	$f_a(p(x_\alpha-x_{\alpha'}))\not\rightarrow 0$, where $f_a$ is the biorthogonal functional of $a$. 
	Then there exist $\epsilon>0$ and a strictly increasing sequence $(\alpha_n)$ of indices such that
	$$
	f_a(p(x_{\alpha_n}-x_{\alpha_{n-1}}))\geq\epsilon>0 \ \ \ (\forall n\in \mathbb N).
	$$
	Thus
	\begin{equation*}
	\begin{split}
	n\epsilon &\leq \sum\limits_{k=2}^{n+1} f_a (p(x_{\alpha_k}-x_{\alpha_{k-1}}))\\
	&=f_a\Big(\sum\limits_{k=2}^{n+1}p(x_{\alpha_k}-x_{\alpha_{k-1}}\Big)=f_a\Big(p\Big(\sum\limits_{k=2}^{n+1}x_{\alpha_k}-x_{\alpha_{k-1}}\Big)\Big)\\ 
	&=f_a(p(x_{\alpha_{n+1}}-x_{\alpha_1}))\leq 2f_a(e).
	\end{split}
	\end{equation*}
	Thus $\displaystyle n\epsilon\leq 2f_a(e)$ for all $n\in\mathbb N$, and hence $\epsilon\leq 0$; a contradiction.
\end{proof}

Remark that the LNVL $(c_0,\lvert\cdot\rvert,\ell_\infty)$ is not $p$-complete, yet the norming lattice $\ell_\infty$ is atomic and its lattice norm is additive on $(c_0)_+$. 
Consider the sequence $x_n=\sum_{i=1}^n e_i$, where $e_n$'s are the standard unit vectors of $c_0$. 
Then $0\le x_n\uparrow$ and $x_n$'s are $p$-bounded by ${\textbf{1}}=(1,1,\cdots)\in\ell_\infty$. 
Clearly, it is not $p$-convergent, so the LNVL $(c_0,\lvert\cdot\rvert,\ell_\infty)$ is not $p$-KB-space. 
Notice also that $(c_0,\lvert\cdot\rvert,\ell_\infty)$ is $op$-continuous.

\subsection{More details on Example \ref{ExLNVL_4}}
Let us discuss Example \ref{ExLNVL_4} in more details.

\begin{itemize}
	\item[(i)] If $X$ is an order continuous Banach lattice, then $(X,p,\ell_{\infty}(B_{X^*}))$ is $op$-continuous.  
	
	\begin{proof}
		Assume $x_\alpha\downarrow 0$, we show $p(x_\alpha)\downarrow 0$. Our claim is the following: $p(x_\alpha)\downarrow 0$ iff $p(x_\alpha)[f]\downarrow 0$ for all $f\in B_{X^*}$.  
		
		For the necessity, let $p(x_\alpha)\downarrow 0$ and $f\in B_{X^*}$. Trivially, $|f|(x_\alpha)\downarrow$. If there exists $z_f\in{\mathbb R}$ such that $0\leq z_f\leq |f|(x_\alpha)$ 
		for all $\alpha$, then
		$$
		0\leq z_f\leq |f|(x_\alpha)\leq\left\|f\right\|\left\|x_\alpha\right\|\downarrow 0.
		$$
		Hence $z_f=0$ and $p(x_\alpha)[f]=|f|(x_\alpha)\downarrow 0$.
		
		For the sufficiency, let $p(x_\alpha)[f]\downarrow 0$ for every $f\in B_{X^*}$. Since $p$ is monotone and $x_\alpha\downarrow$, 
		then $p(x_\alpha)\downarrow$. If $0\leq\varphi\leq p(x_\alpha)$ for all $\alpha$, then 
		$$
		0 \leq \varphi (f) \leq p(x_\alpha) [f]=|f|(x_\alpha)  \ \ \ (\forall f \in B_{X^*}).
		$$ 
		So by the assumption, we get $\varphi(f)=0$ for all $f\in B_{X^*}$, and hence $\varphi=0$. 
		Therefore, $p(x_\alpha) \downarrow 0$. $\square$
	\end{proof}
	
	\item[(ii)] If $(X,\left\|\cdot\right\|)$ is a KB-space, then $(X,p,E)$ is a $p$-KB-space. 
	
	\begin{proof}
		Suppose that $0\leq x_\alpha\uparrow$ and $p(x_\alpha)\leq\varphi\in \ell_\infty(B_{X^*})$. As
		\begin{eqnarray*}
			\left\|x_\alpha\right\|&=&\sup_{f\in B_{X^*}}|f(x_\alpha)|\leq\sup_{f\in B_{X^*}}|f|(x_\alpha) \\ 
			&=& \sup_{f\in B_{X^*}}p(x_\alpha)[f]\leq\varphi[f]\leq\|\varphi\|_{\infty}\le\infty \ \ \ (\forall\alpha),
		\end{eqnarray*}
		and since $X$ is a KB-space, we get $\left\|x_\alpha-x\right\|\to 0$ for some $x\in X_+$. 
		So, for any $f\in B_{X^*}$, we have $|f|(|x_\alpha-x|)\rightarrow 0$ or $p(x_\alpha-x)[f]\rightarrow 0$. 
		Thus $p(x_\alpha-x)\xrightarrow[]{o} 0$ in $\ell_\infty (B_{X^*})$ and hence $x_\alpha\xrightarrow{p}x$.
	\end{proof}
\end{itemize}

Recall that a vector lattice $X$ is called {\em perfect} if the natural embedding from $X$ into $(X_n^\sim)_n^\sim$ 
is one-to-one and onto, where $X_n^\sim$ denotes the order continuous dual of $X$ \cite[pp.58-59]{AB}. 
If $X$ is a perfect vector lattice, then $X_n^\sim$ separates the points of $X$ 
\cite[Thm.5.6(1)]{AB}. 

\begin{proposition}
	Let $X$ be a perfect vector lattice, $Y=X_n^\sim$ and $p:X\to\mathbb{R}^Y$ be defined as $p(x)[f]:=|f|(|x|)$, where $f\in Y$. 
	Then the LNVL $(X,p,\mathbb{R}^Y)$ is a $p$-KB-space.
\end{proposition}

\begin{proof}
	Assume $0\leq x_\alpha \uparrow$ in $X$ and $p(x_\alpha)\leq\varphi\in\mathbb{R}^Y$. 
	Then, for all $f\in Y$, we have $p(x_\alpha)[f]\leq\varphi(f)$ or $|f|(x_\alpha)\leq\varphi(f)$. 
	So for all $f\in Y$, $\sup\limits_\alpha|f|(x_\alpha)<\infty$, and hence, by \cite[Thm.5.6(2)]{AB}, 
	there is $x \in X$ with $x_\alpha\uparrow x$. An argument similar to $(i)$ above shows that $X$ is $op$-continuous. 
	Therefore, $x_\alpha\xrightarrow{p}x$.
\end{proof}

\subsection{$p$-Fatou space}
In this subsection, we introduce and discuss $p$-Fatou spaces.

\begin{definition}
	An LNVL $(X,p,E)$ is called {\em $p$-Fatou space} if $0\le x_\alpha\uparrow x$ in $X$ implies $p(x_\alpha)\uparrow p(x)$. 
\end{definition}

Note that $(X,p,E)$ is a $p$-Fatou space iff $p$ is order semicontinuous \cite[2.1.4, p.48]{K}. 
Clearly any $op$-continuous LNVL $(X,p,E)$ is a $p$-Fatou space. 
It is easy to see that the LNVL $(c,p,c)$ in Example \ref{ExLNVL_1} is not a $p$-Fatou space. Moreover 
the $p$-Fatou property ensures that $p$-bands are bands.

\begin{proposition}\label{bands}
	Let $B$ be a $p$-band in a $p$-Fatou space $(X,p,E)$. Then $B$ is a band in $X$.
\end{proposition}

\begin{proof}
	Let $B=M^{{\perp_{\mathbf p}}}=\{x\in X:(\forall m\in M)\ p(x)\bot p(m)\}$ for some nonempty $M\subseteq X$.
	Since $B$ is an ideal in $X$ to show that $B$ is a band it is enough to prove that
	if $B_+\ni b_\alpha\uparrow x\in X$, then $x\in B$. That is easy, since $p(b_\alpha)\uparrow p(x)$ 
	as $X$ is a $p$-Fatou space. By $o$-continuity of lattice operations in $E$, we obtain that
	$$
	0=p(b_\alpha)\wedge p(m)\xrightarrow{o} p(x)\wedge p(m) \ \ \ \ (\forall m \in M).
	$$
	Therefore, $p(x)\wedge p(m)=0$ for all $m\in M$, and hence $x\in B$. 
\end{proof}

In connection with Proposition \ref{bands} and Example \ref{ExLNVL_1}, the following open problem arises.

\begin{problem}
	Let $(X,p,E)$ be a decomposable LNVL $p$-Fatou space in which every $p$-band is a band. 
	Is $X$ a $p$-Fatou space?
\end{problem}

\subsection{$p$-Density and $p$-units} 
In the present paper, we use the following definition of a $p$-dense subset in an LNS, which is motivated by the notion of a dense subset of a normed space. 

\begin{definition}
	Given an LNS $(X,p,E)$ and $A\subseteq X$. A subset $B\subseteq A$ is said to be {\em $p$-dense in $A$} 
	if, for any $a\in A$ and for any $0\ne u\in p(X)$ there is $b\in B$ such that $p(a-b)\le u$.
\end{definition} 

\begin{remark}\ \
	\begin{enumerate}
		\item Consider the LNVL $(X,p,E)$ with $p=\lvert\cdot\rvert$, $E=X$, and let $Y$ be a sublattice $X$. If $Y$ is $p$-dense in $X$, then $Y$ is order dense. Indeed, let $0\neq x\in X_+$, 
		then there is $y\in Y$ such that $|y-\frac{1}{2}x|\leq\frac{1}{3}x$ which implies $0<\frac{1}{6}x\leq y\leq\frac{5}{6}x$, and so $0<y\leq x$.
		
		\item $c$ is order dense, yet is not $p$-dense in both of the following LNVLs: $(\ell_\infty,\lVert\cdot\rVert_\infty,{\mathbb R})$ and $(\ell_\infty,\lvert\cdot\rvert,\ell_\infty)$.
		
		\item If $X=(X,\lVert\cdot\rVert)$ is a normed lattice, $p=\lVert\cdot\rVert$ and $E=\mathbb{R}$, then clearly a subset $Y$ of $X$ is $p$-dense iff $Y$ is norm dense.
	\end{enumerate}
\end{remark}

The following notion is motivated by the notion of a weak order unit in a vector lattice $X=(X,\lvert\cdot\rvert,X)$ and by the notion of a quasi-interior point in 
a normed lattice $X=(X,\lVert\cdot\rVert,{\mathbb R})$

\begin{definition}\label{$p$-unit}
	Let $(X,p,E)$ be an LNVL. A vector $e\in X$ is called a {\em $p$-unit} if, for any $x\in X_+$ we have $p(x-x\wedge ne)\xrightarrow{o}0$.
\end{definition}

\begin{remark}\label{properties of $p$-units}
	Let $(X,p,E)$ be an LNVL.
	\begin{enumerate}
		\item If $X\neq\{0\}$ then, for any $p$-unit $e$ in $X$ it holds that $e>0$. 
		Indeed, let $e$ be a $p$-unit in $X\ne\{0\}$. Trivially, $e\ne 0$. Suppose $e^->0$. Then, for $x:=e^-$, we obtain that
		$$
		p(x-x\wedge ne)=p(e^--(e^-\wedge n(e^+-e^-)))=
		$$
		$$  
		p(e^--(e^-\wedge n(-e^-)))=p(e^--(-ne^-))=p((n+1)e^-)=
		$$
		$$  
		(n+1)p(e^-)\not\xrightarrow{o}0  
		$$
		as $n\to\infty$. This is impossible, because $e$ is a $p$-unit. Therefore, $e^-=0$ and $e>0$. 
		
		\item Let $e\in X$ be a $p$-unit. Given $0<\alpha\in\mathbb{R}_+$ and $z\in X_+$. Observe that, for $x \in X_+$, 
		$p(x-n\alpha e\wedge x)=\alpha p(\frac{x}{\alpha}-ne\wedge\frac{x}{\alpha})$ and 
		$p(x-n(e+z)\wedge x)\leq p(x-x\wedge ne)$, from which it follows easily that $\alpha e$ and $e+z$ are both $p$-units.
		
		\item If $e\in X$ is a strong unit, then $e$ is a $p$-unit. Indeed, let $x\in X_+$, then there is $k\in\mathbb{N}$ such that $x\leq ke$, so $x-x\wedge ne=0$ for any $n\geq k$.
		
		\item If $e\in X$ is a $p$-unit, then $e$ is a weak unit. Assume $x\wedge e =0$, then $x\wedge ne =0$ for any $n\in\mathbb{N}$. Since $e$ is a $p$-unit, then $p(x)=0$ and hence $x=0$.
		
		\item If $X$ is $op$-continuous, then clearly every weak unit of $X$ is a $p$-unit.
		
		\item In $X=(X,\lvert\cdot\rvert,X)$, the lattice norm $p(x)=|x|$ is always order continuous. Therefore, the notions of $p$-unit and of weak unit coincide in $X$.
		
		\item If $X=(X,\lVert\cdot\rVert)$ is a normed lattice, $p=\lVert\cdot\rVert$, $E=\mathbb{R}$, and $e\in X$, then $e$ is a $p$-unit iff $e$ is a quasi-interior point of $X$.
	\end{enumerate}
\end{remark}

In the proof of the following proposition, we use the same technique as in the proof of \cite[Lm.4.15]{AA}.

\begin{proposition}
	Let $(X,p,E)$ be an LNVL, $e\in X_+$, and $I_e$ be the order ideal generated by $e$ in $X$. 
	If $I_e$ is $p$-dense in $X$, then $e$ is a $p$-unit.
\end{proposition}

\begin{proof}
	Let $0\neq u\in p(X)$. Let $x\in X_+$, then there exists $y\in I_e$ such that $p(x-y)\leq u$. 
	Since $|y^+\wedge x-x|\leq |y^+-x|=|y^+-x^+|\leq |y-x|$, then, by replacing $y$ by $y^+\wedge x$, 
	we may assume without loss of generality that there is $y\in I_e$ such that $0\leq y\leq x$ and $p(x-y)\leq u$. 
	Thus, for any $m\in\mathbb{N}$, there is $y_m\in I_e$ such that $0\leq y_m\leq x$ and 
	$$
	p(x-y_m) \leq \frac{1}{m}u.
	$$
	Since $y_m\in I_e$, then there exists $k=k(m)\in\mathbb{N}$ such that $0\leq y_m\leq ke$, and so $0\leq y_m\leq ke\wedge x$. 
	
	For $n\geq k$, $x-x\wedge ne\leq x-x\wedge ke\leq x-y_m$, and so $p(x-x\wedge ne)\leq p(x-y_m)\leq\frac{1}{m}u$. 
	Hence $p(x-x\wedge ne)\xrightarrow{o}0$. Thus $e$ is a $p$-unit.
\end{proof}

\section{Unbounded $p$-convergence}
The {\em $up$-convergence} in LNVLs generalizes the $uo$-convergence in vector lattices 
\cite{GTX,G,GX}, the $un$-convergence \cite{DOT} and the $uaw$-convergence \cite{Z} in Banach lattices. 
We study basic properties of the $up$-convergence and characterize the $up$-convergence in certain LNVLs. 

\subsection{Main definition and its motivation} 
Let $(X,p,E)$ be an LNVL. The following definition is motivated by its special case when it is reduced to the $un$-convergence for a
normed lattice $(X,p,E)=(X,\|\cdot\|,{\mathbb R})=(X,\|\cdot\|)$.

\begin{definition}\label{$up$-convergence}
	A net $x_\alpha\subseteq X$ is said to be unbounded $p$-convergent to $x\in X$ $($shortly, $x_\alpha$ $up$-converges to $x$ or  
	$x_\alpha\xrightarrow{up}x$$)$, if
	$$
	p(|x_\alpha-x|\wedge u)\xrightarrow{o}0 \ \ \ (\forall u\in X_+).
	$$
\end{definition}

It is immediate to see that $up$-convergence coincides with $un$-convergence in the case when $p$ is the norm in a normed lattice, and 
with $uo$-convergence in the case when $X=E$ and $p(x)=|x|$. It is clear that $x_{\alpha}\xrightarrow[]{p} x$ implies $x_\alpha\xrightarrow{up}x$, and for order bounded nets
$up$-convergence and $p$-convergence agree. It should be also clear that, if an LNVL $X$ is $op$-continuous, then $uo$-convergence in $X$ implies $up$-convergence. 
The $uaw$-convergence is also a particular case of $up$-convergence as it follows from the next proposition.

\begin{proposition}\label{Zabeti}
	In the notation of Example \ref{ExLNVL_5}, $x_\alpha\xrightarrow{up}0$ in $X$ iff for every $u\in X_+$, $|x_\alpha|\wedge u\xrightarrow{|\sigma|(X,Y)}0$.
\end{proposition}

\begin{proof} 
	$x_\alpha \xrightarrow{up}0$ in $X$ iff for all $u\in X_+$, $p(|x_{\alpha}|\wedge u)\xrightarrow{o}0$ in $E$ 
	iff for every $u\in X_+$, $p(|x_\alpha|\wedge u)[y]\to 0$ for all $y\in Y$ iff for every $u\in X_+$, $|y|(|x_\alpha|\wedge u)\to 0$ for all $y\in Y$
	iff for every $u\in X_+$, $|x_\alpha|\wedge u\xrightarrow{|\sigma|(X,Y)}0$. 
\end{proof}

In particular, if $X$ is a Banach lattice, $Y=X^*$, the topological dual of $X$, $E = {\mathbb R}^Y$ and $p: X \to E$ as defined above,
then $x_\alpha \xrightarrow{up}0$  in $X$ iff $x_\alpha \xrightarrow{uaw} 0$.

\subsection{Basic results on $up$-convergence}
We begin with the next list of properties of $up$-convergence which follows directly from Lemma \ref{LO are $p$-continuous}.

\begin{lemma}
	Let $x_\alpha\xrightarrow{up}x$ and $y_\alpha\xrightarrow{up}y$ in an LNVL $(X,p,E)$, then$:$
	
	$(i)$\ \ \ $ax_\alpha+by_\alpha\xrightarrow{up} ax+by$ for any $a,b\in{\mathbb R}$$,$
	in particular, if $x_\alpha=y_\alpha$, then $x=y$$;$
	
	$(ii)$\ \  $x_{\alpha_\beta}\xrightarrow{up}x$ for any subnet $x_{\alpha_\beta}$ of $x_\alpha$$;$
	
	$(iii)$\   $|x_\alpha|\xrightarrow{up} |x|$$;$
	
	$(iv)$\ \  if $x_\alpha\ge y_\alpha$ for all $\alpha$, then $x\ge y$.
\end{lemma} 

\begin{lemma}\label{$up$-sup-inf}
	Let $x_\alpha$ be a monotone net in an LNVL $(X,p,E)$ such that $x_\alpha\xrightarrow[]{up}x$, then $x_\alpha\xrightarrow[]{o}x$.
\end{lemma}

\begin{proof}
	The proof of Proposition \ref{$p$-sup-inf} is applicable here as well.
\end{proof}

The following result is a $p$-generalization of \cite[Lm.1.2 (ii)]{KMT}.

\begin{theorem}
	Let $x_\alpha$ be a monotone net in an LNVL $(X,p,E)$ which $up$-converges to $x$. Then $x_\alpha\xrightarrow[]{p}x$.
\end{theorem}

\begin{proof}
	Without loss of generality we may assume that $0\le x_\alpha\uparrow$ \ . From Lemma \ref{$up$-sup-inf} it follows that 
	$0\le x_\alpha\uparrow x$ for some $x\in X$. So $0\leq x-x_\alpha\leq x$ for all $\alpha$. 
	Since, for each $u\in X_+$, we know that
	$$
	p((x-x_\alpha)\wedge u)\xrightarrow[]{o} 0.
	$$ 
	In particular, for $u=x$, we obtain that
	$$
	p(x-x_\alpha)=p((x_\alpha-x)\wedge x)\xrightarrow[]{o} 0.
	$$
\end{proof}	

Similar to \cite[Lm.1.2.(1)]{G} we have that if $x_{\alpha}\xrightarrow[]{up} 0$ in an LNVL $(X,p,E)$, then $\inf_{\beta}|y_{\beta}|=0$ for any subnet $y_{\beta}$ of 
the net $x_{\alpha}$. Indeed, let $y_{\beta}$ be a subnet of $x_{\alpha}$. Clearly, $y_\beta\xrightarrow[]{up} 0$. 
If $0\leq z\leq |y_\beta|$ for all $\beta$, then $p(z)=p(z\wedge |y_\beta|)\xrightarrow[]{o} 0$, and so $z=0$. Hence $\inf_{\beta}|y_{\beta}|=0$.\\

The following two results, which are analogies of Lemma 2.8 in \cite{DOT} and of Lemma 3.6 in \cite{GX}, we have respectively. 

\begin{lemma}
	Let $(X,p,E)$ be an LNVL. Assume that $E$ is order complete and $x_\alpha\xrightarrow[]{up} x$, then $p(|x|-|x|\wedge|x_\alpha|)\xrightarrow[]{o} 0$ and 
	$p(x)=\liminf_\alpha p(|x|\wedge |x_\alpha|)$. Moreover, if $x_\alpha$ is $p$-bounded, then $p(x)\leq\liminf_{\alpha}p(x_\alpha)$.
\end{lemma}

\begin{proof}
	Note that
	$$
	|x|-|x|\wedge |x_\alpha|=|\,|x_\alpha|\wedge |x|-|x|\wedge |x|\,|\leq |\,|x_\alpha|-|x|\,|\wedge |x|\leq |x_\alpha-x|\wedge |x|.
	$$
	Since $x_\alpha\xrightarrow[]{up} x$, we get $p(|x|-|x|\wedge |x_\alpha|)\xrightarrow[]{o} 0$. Thus 
	$$
	p(x)=p(|x|)\leq p(|x|-|x|\wedge |x_\alpha|)+p(|x|\wedge |x_\alpha|).
	$$ 
	So $p(x)\leq\liminf_\alpha p(|x|\wedge |x_\alpha|)$. Hence $p(x)=\liminf_\alpha p(|x|\wedge |x_\alpha|)$.
\end{proof}

\begin{lemma}
	Let $(X,p,E)$ be an $op$-continuous LNVL. Assume that $E$ is order complete and $x_\alpha\xrightarrow[]{uo} x$, then $p(|x|-|x|\wedge|x_\alpha|)\xrightarrow[]{o} 0$ 
	and $p(x)=\liminf_\alpha p(|x|\wedge|x_\alpha|)$. Moreover, if $x_\alpha$ is $p$-bounded, then $p(x)\leq\liminf_{\alpha}p(x_\alpha)$.
\end{lemma}

We finish this subsection with the following technical lemma.

\begin{lemma} \label{ptechlemma2}
	Given an LNVL $(X, p,E)$. If $x_\alpha\xrightarrow{p}x$ and $x_\alpha$ is an $o$-Cauchy net, then $x_\alpha\xrightarrow{o}x$. 
	Moreover, if $x_\alpha\xrightarrow{p}x$ and $x_\alpha$ is $uo$-Cauchy, then $x_\alpha\xrightarrow{uo}x$.
\end{lemma}

\begin{proof}
	Since $x_\alpha$ is order Cauchy, then $x_\alpha-x_\beta\xrightarrow{o}0$ as $\alpha,\beta\to\infty$. 
	So there exists $z_\gamma\downarrow 0$ such that, for every $\gamma$, there exists $\alpha_\gamma$ satisfying 
	\begin{equation} \label{techlemma2eqn}
	|x_\alpha-x_\beta|\leq z_\gamma, \,\,\,\, \forall\alpha,\beta\geq \alpha_\gamma.
	\end{equation}
	By taking limit over $\beta$ in (\ref{techlemma2eqn}) and applying Lemma \ref{LO are $p$-continuous}, we get 
	$|x_\alpha-x|\leq z_\gamma$ for all $\alpha\geq\alpha_\gamma$. Thus $x_\alpha\xrightarrow{o} x$.
	
	For the $uo$-convergence, the similar argument is used, so the proof is omitted.
\end{proof}

\subsection{$up$-Convergence and $p$-units}
The following result is a generalization of \cite[Lm.2.11]{DOT} and of \cite[Cor.3.5]{GTX} in LNVLs. 

First of all, we recall useful characterizations of order convergence.
For any order bounded net $x_\alpha$ in an order complete vector lattice $E$, $x_\alpha\xrightarrow[]{o} x$ iff $\limsup \limits_\alpha|x_\alpha-x|=\inf_\alpha\sup_{\beta\geq\alpha}|x_\beta-x|=0$. 
Moreover, for any net $x_\alpha$ in a vector lattice $E$, $x_\alpha\xrightarrow[]{o} 0$ in $E$ iff $x_\alpha\xrightarrow[]{o} 0$ in $E^\delta$ (the order completion of $E$); see, e.g., \cite[Cor.2.9]{GTX}.

\begin{theorem}\label{$up$-conv by $p$-unit}
	Let $(X,p,E)$ be an LNVL and $e\in X_+$ be a $p$-unit. Then $x_{\alpha}\xrightarrow{up}0$ iff $p(|x_\alpha|\wedge e)\xrightarrow{o} 0$ in $E$.
\end{theorem}

\begin{proof}
	The ``only if'' part is trivial. For the ``if'' part, let $u \in X_+$, then
	$$
	|x_\alpha|\wedge u\leq |x_\alpha|\wedge(u-u\wedge ne)+|x_\alpha|\wedge(u\wedge ne) 
	\leq (u-u\wedge ne)+n(|x_\alpha|\wedge e), 
	$$ 
	and so 
	$$ 
	p(|x_\alpha| \wedge u) \leq p(u-u \wedge ne) + np(|x_\alpha| \wedge e)
	$$ 
	holds in $E^\delta$ for any $\alpha$ and any $n\in\mathbb{N}$. Hence 
	$$
	\limsup\limits_\alpha p(|x_\alpha|\wedge u)\leq p(u-u\wedge ne)+n\limsup\limits_\alpha p(|x_\alpha|\wedge e)
	$$ 
	holds in $E^\delta$ for all $n\in\mathbb{N}$. Since $p(|x_\alpha|\wedge e)\xrightarrow{o}0$ in $E$, 
	then $p(|x_\alpha|\wedge e)\xrightarrow{o}0$ in $E^\delta$, and so $\limsup\limits_\alpha p(|x_\alpha|\wedge e)=0$ in $E^\delta$. Thus
	$$
	\limsup\limits_\alpha  p(|x_\alpha|\wedge u)\leq p(u-u\wedge ne)
	$$
	holds in $E^\delta$ for all $n\in\mathbb{N}$. Since $e$ is a $p$-unit, we have that $\limsup\limits_\alpha p(|x_\alpha|\wedge u)=0$ 
	in $E^\delta$ or $p(|x_\alpha|\wedge u)\xrightarrow[]{o} 0$ in $E^\delta$. It follows that $p(|x_\alpha|\wedge u)\xrightarrow[]{o} 0$ in $E$ and hence $x_\alpha\xrightarrow[]{up} 0$. 
\end{proof}

\subsection{$up$-Convergence and sublattices}
Given an LNVL $(X,p,E)$, a sublattice $Y$ of $X$, and a net $(y_\alpha)_\alpha\subseteq Y$. Then $y_\alpha\xrightarrow[]{up} 0$ in $Y$ has the meaning that
$$
p(|y_\alpha|\wedge u)\xrightarrow[]{o} 0 \ \ \ \  (\forall u\in Y_+).
$$
\

The following lemma is a $p$-version of \cite[Lm.3.3]{GX}.

\begin{lemma}
	Let $(X,p,E)$ be the LNVL, $B$ be a projection band of $X$, and $P_B$ be the corresponding band projection. 
	If $x_\alpha\xrightarrow{up}x$ in $X$, then $P_B(x_\alpha)\xrightarrow{up} P_B(x)$ in both $X$ and $B$. 
\end{lemma}

\begin{proof}
	It is known that $P_B$ is a lattice homomorphism and $0\leq P_B\leq I$. Since $|P_B(x_\alpha)-P_B(x)|=P_B|x_\alpha-x|\leq |x_\alpha-x|$, 
	then it follows easily that $P_B(x_\alpha)\xrightarrow{up} P_B(x)$ in both $X$ and $B$.
\end{proof}

Let $(X,p,E)$ be an LNVL and $Y$ be a subset of $X$. Then $Y$ is called {\em $up$-closed} in $X$ if, for any net $y_\alpha$ in $Y$ that is $up$-convergent to $x\in X$, we have $x\in Y$. 
Clearly, every band is $up$-closed.\\ 

We present a $p$-version of \cite[Prop.3.15]{GTX} with a similar proof.

\begin{proposition}
	Let $X$ be an LNVL and $Y$ be a sublattice of $X$. Suppose that either $X$ is $op$-continuous or $Y$ is a $p$-KB-space in its own right.  
	Then $Y$ is $up$-closed in $X$ iff it is $p$-closed in $X$.
\end{proposition}

\begin{proof}
	Only the sufficiency requires a proof. Let $Y$ be $p$-closed in $X$ and $y_\alpha$ be a net in $Y$ with $y_\alpha\xrightarrow{up} x\in X$.
	WLOG, we assume $(y_\alpha)_\alpha\subseteq Y_+$ because the lattice operations in $X$ are $p$-continuous. 
	Note that, for every $z\in X_+$, $|y_\alpha\wedge z-x\wedge z|\leq |y_\alpha-x|\wedge z$ \ (cf. the inequality (1) in the proof of \cite[Prop.3.15]{GTX}).
	So $p(y_\alpha\wedge z-x\wedge z)\le p(|y_\alpha-x|\wedge z)\xrightarrow{o}0$.
	In particular, $Y\ni y_\alpha\wedge y\xrightarrow{p} x\wedge y$ in $X$ for any $y\in Y_+$. 
	Since $Y$ is $p$-closed, $x\wedge y\in Y$ for any $y\in Y_+$. Since, for any $0\le z\in Y^\perp$ 
	and for any $\alpha$ we have $y_\alpha\wedge z=0$, then 
	$$
	|x\wedge z|=|y_\alpha\wedge z-x\wedge z|\le |y_\alpha-x|\wedge z\xrightarrow{p} 0. 
	$$
	Therefore, $x\wedge z=0$, and hence $x\in Y^{\perp\perp}$. Since $Y^{\perp\perp}$ is the band generated by $Y$ in $X$, there is a net $z_\beta$ in the ideal
	$I_Y$ generated by $Y$ such that $0\le z_\beta\uparrow x$ in $X$. Take for every $\beta$ an element $w_\beta\in Y$ with $z_\beta\le w_\beta$.
	Then $x\ge w_\beta\wedge x\ge z_\beta\wedge x=z_\beta\uparrow x$ in $X$, and so $w_\beta\wedge x\xrightarrow{o} x$ in $X$.
	
	Case 1: If $X$ is $op$-continuous, then $w_\beta\wedge x\xrightarrow{p} x$. Since $w_\beta\wedge x\in Y$ and $Y$ is $p$-closed, we get $x\in Y$.
	
	Case 2: Suppose $Y$ is a $p$-KB-space in its own right. Let $\Delta$ be the collection of all finite subsets of the index set $B$. For each $\delta=\{\beta_1,\ldots,\beta_n\} \in \Delta$ let $y_\delta:=(w_{\beta_1}\vee \ldots \vee w_{\beta_n})\wedge x$. Clearly, $y_\delta \in Y$, $0\leq y_\delta \uparrow$, and the net $(y_\delta)$ is $p$-bounded in $Y$. Since $Y$ is a $p$-KB-space, then there is $y_0 \in Y$ such that $y_\delta \xrightarrow{p} y_0$ in $Y$ and trivially in $X$. Since $(y_\delta)$ is monotone then it follows from Proposition \ref{$p$-sup-inf} that $y_\delta \uparrow y_0$ in $X$. Also, we have $y_\delta \xrightarrow{o} x$ in $X$. Thus, $x=y_0 \in Y$.
\end{proof}

\subsection{$p$-Almost order bounded sets} 
Recall that a subset $A$ in a normed lattice $(X,\|\cdot\|)$ is said to be almost order bounded if, for any $\epsilon>0$, there is $u_\epsilon\in X_+$ 
such that $\|(|x|-u_\epsilon)^+\|=\||x|-u_\epsilon\wedge|x|\|\leq\epsilon$ for any $x\in A$. Similarly we have:

\begin{definition}
	Given an LNVL $(X,p,E)$. A subset $A$ of $X$ is called a {\em $p$-almost order bounded} if, for any $w\in E_+$, there is $x_w\in X_+$ such that $p((|x|-x_w)^+)=p(|x|-x_w\wedge |x|)\leq w$ for any $x\in A$.
\end{definition}

It is clear that $p$-almost order boundedness notion in LNVLs is a generalization of almost order boundedness in normed lattices. 
In the LNVL $(X,|\cdot|,X)$, a subset of $X$ is $p$-almost order bounded, iff it is order bounded in $X$. 

The following result is a $p$-version of \cite[Lm.2.9]{DOT}, and it is also similar to \cite[Prop.3.7]{GX}.

\begin{proposition}\label{$p$-almost order boundedness and $up$-convergence}
	If $(X,p,E)$ is an LNVL, $x_\alpha$ is $p$-almost order bounded, and $x_\alpha\xrightarrow{up} x$, then $x_\alpha\xrightarrow{p} x$.
\end{proposition}

\begin{proof}
	Since $x_\alpha$ is $p$-almost order bounded, then it is easy to see that the net $(|x_\alpha-x|)_\alpha$ is also $p$-almost order bounded. 
	So given $w\in E_+$. Then there exists $x_w\in X_+$ with 
	$$ 
	p(|x_\alpha-x|-|x_\alpha-x|\wedge x_w)\leq w
	$$
	But $x_\alpha \xrightarrow{up}x$, so $\limsup\limits_\alpha p(|x_\alpha-x|\wedge x_w)=0$ in $E^\delta$. Thus, for any $\alpha$,
	$$ 
	p(x_\alpha -x)=p(|x_\alpha -x|) \leq p(|x_\alpha -x|-|x_\alpha -x|\wedge x_w)+p(|x_\alpha -x|\wedge x_w)\leq w + p(|x_\alpha -x|\wedge x_w) 
	$$
	Hence
	$$
	\limsup\limits_\alpha p(x_\alpha -x)\leq w+\limsup\limits_\alpha p(|x_\alpha -x|\wedge x_w)\leq w
	$$
	holds in $E^\delta$. But $w\in E_+$ is arbitrary, so $\limsup \limits_\alpha p(x_\alpha -x)=0$ in $E^\delta$. Thus $p(x_\alpha -x)\xrightarrow{o}0$ in $E^\delta$, and so in $E$.
\end{proof}

The following proposition is a $p$-version of \cite[Prop.4.2]{GX}.

\begin{proposition}
	Given an $op$-continuous and p-complete LNVL $(X,p,E)$. Then every $p$-almost order bounded $uo$-Cauchy net is $uo$- and $p$-convergent to the same limit.
\end{proposition}

\begin{proof}
	Let $x_\alpha$ be a $p$-almost order bounded $uo$-Cauchy net. Then the net $(x_\alpha-x_{\alpha'})$ is $p$-almost order bounded and is $uo$-converges to $0$. 
	Since $X$ is $op$-continuous, then $x_\alpha-x_{\alpha'}\xrightarrow{up}0$ and, by Proposition \ref{$p$-almost order boundedness and $up$-convergence}, 
	we get $x_\alpha-x_{\alpha'}\xrightarrow{p}0$. Thus $x_\alpha$ is $p$-Cauchy, and so is $p$-convergent. By Lemma \ref{ptechlemma2}, 
	we get that $x_\alpha$ is also $uo$-convergent to its $p$-limit.
\end{proof}

\subsection{$rup$-Convergence}
In this subsection, we introduce the notions of $rup$-convergence and of an $rp$-unit. Recall that a net $(x_{\alpha})_{\alpha\in A}$ in a vector lattice $E$ is \emph{relatively uniform convergent} 
(or {\em $ru$-convergent}, for short) to $x\in E$ if there is $y\in E_+$,  such that, for any $\varepsilon>0$, there exists $\alpha_0\in A$ such that $|x_\alpha-x|\leq\varepsilon y$ 
for any $\alpha\geq\alpha_0$, \cite[Thm.16.2]{LZ1}. In this case we write $x_{\alpha}\xrightarrow{ru}x$.

\begin{definition}
	Let $(X,p,E)$ be an LNVL. A net $(x_\alpha)_\alpha\subseteq X$ is said to be {\em relatively unbounded $p$-convergent} $(${\em $rup$-convergent}$)$ to $x\in X$ if 
	$$
	p(|x_\alpha-x|\wedge u)\xrightarrow{ru}0 \ \ \ (\forall u\in X_+).
	$$ 
	In this case we write $x_\alpha\xrightarrow{rup} x$.
\end{definition}

Clearly, $rup$-convergence implies $up$-convergence, but the converse need not be true.

\begin{definition}\label{rp-unitdef}
	Given an LNVL $(X,p,E)$. A vector $e\in X$ is called an $rp$-unit if, for any $x\in X_+$, we have $p(x-x\wedge ne) \xrightarrow{ru} 0$.
\end{definition}

Obviously, every $rp$-unit is a $p$-unit. So, by Remark \ref{properties of $p$-units} $(1)$ after Definition \ref{$p$-unit}, if $e\in X\neq\{0\}$ is an $rp$-unit then $e>0$.
Not every $p$-unit is an $rp$-unit. To see this, take $X=(C_b({\mathbb R}),\lvert\cdot\rvert,C_b({\mathbb R}))$ and $e=e(t)=e^{-|t|}$. 
Then $e$ is a $p$-unit. However, $e$ is not an $rp$-unit since $p(1-1\wedge ne)$ does not $ru$-converge to 0, where $1(t)\equiv 1$.

The following result generalizes \cite[Cor.3.5]{GTX} and \cite[Lm.2.11]{DOT}.

\begin{proposition}
	Let $(X,p,E)$ be an LNVL with an $rp$-unit $e$. Then $x_{\alpha}\xrightarrow{rup}0$ iff $p(|x_\alpha|\wedge e)\xrightarrow{ru}0$.
\end{proposition}

\begin{proof}
	The ``only if'' part is trivial. For the ``if'' part let $u\in X_+$, then
	$$
	|x_\alpha|\wedge u\leq|x_\alpha|\wedge(u-u\wedge ne)+|x_\alpha|\wedge(u\wedge ne)\leq(u-u\wedge ne)+n(|x_\alpha|\wedge e), 
	$$
	and so
	$$ 
	p(|x_\alpha|\wedge u)\leq p(u-u\wedge ne)+np(|x_\alpha|\wedge e)
	$$
	holds for any $\alpha$ and any $n\in\mathbb{N}$. 
	
	Given $\varepsilon>0$. Since $e$ is an $rp$-unit, then there is $y\in E_+$ and $n_0\in\mathbb{N}$ such that 
	$$ 
	p(u-u\wedge n_0 e)\leq\frac{\varepsilon}{2}y.
	$$
	It follows from $p(|x_\alpha|\wedge e)\xrightarrow{ru}0$ that there exists $z\in E_+$ and $\alpha_0$ such that 
	$$ 
	p(|x_\alpha|\wedge e)\leq\frac{\varepsilon}{2n_0} z
	$$
	for any $\alpha\geq\alpha_0$. Take $w:=y\vee z$, then 
	$$
	p(|x_\alpha|\wedge u)\leq\varepsilon w
	$$
	for any $\alpha\geq\alpha_0$. Therefore, $p(|x_\alpha|\wedge u)\xrightarrow{ru}0$.
\end{proof}

\section{$up$-Regular sublattices}

The $up$-convergence passes obviously to any sublattice of $X$. As it was remarked in \cite[p.3]{DOT}, in opposite to $uo$-convergence \cite[Thm.3.2]{GTX},
the $un$-convergence does not pass even from regular sublattices. These two facts motivate the following notion in LNVLs.

\begin{definition}
	Let $(X,p,E)$ be an LNVL and $Y$ be a sublattice of $X$. Then $Y$ is called {\em $up$-regular} if, for any net $y_\alpha$ in $Y$, 
	$y_\alpha\xrightarrow{up}0$ in $Y$ implies $y_\alpha\xrightarrow{up}0$ in $X$. Equivalently, $Y$ is $up$-regular 
	in $X$ when $y_\alpha\xrightarrow{up}0$ in $Y$ iff $y_\alpha\xrightarrow{up}0$ in $X$ for any net $y_\alpha$ in $Y$.
\end{definition}

It is clear that if $Y$ is a regular sublattice of a vector lattice $X$, then $Y$ is $up$-regular
in the LNVL $(X, \lvert \cdot \rvert)$; see \cite[Thm.3.2]{GTX}. The converse does not hold in
general.

\begin{example}
	Let $X=B([0,1])$ be the space of all real-valued bounded functions on $[0,1]$ and $Y=C[0,1]$. First of all $X$  under the pointwise ordering $($i.e., $f\leq g$ in $X$ iff $f(t)\leq g(t)$ for all $t\in [0,1]$$)$ is a vector lattice and if we equip $X$  with the $\infty$-norm, then it becomes a Banach lattice.\\
	We claim that the sublattice $Y=(Y,\lvert \cdot \rvert,Y)$ is a $up$-regular sublattice of $X=(X,\lvert \cdot \rvert,X)$. Let $ (f_\alpha) $ be a net in $Y$ such that $f_\alpha \xrightarrow{up} 0$ in  $Y$. That is $\lvert f_\alpha \rvert \wedge g \xrightarrow{o} 0$ in $X$ for any $g \in Y_+$. In particular, we have $\lvert f_\alpha \rvert \wedge \textbf{1} \xrightarrow{o} 0$ in $X$, where $ \textbf{1} $ denotes the constant function one. Since $ \textbf{1} $ is a strong unit in $X$, then it is a $p$-unit for the LNVL $(X,\lvert \cdot \rvert,X)$. It follows from Theorem \ref{$up$-conv by $p$-unit} in Subsection 5.1.3 that $f_\alpha \xrightarrow{up}  0$ in  $X$.
	However, the sublattice $Y$ is not regular in $X$. Indeed, for each $n \in \mathbb{N}$ let $f_n$ be a continuous function on $[0,1]$ defined as: 
	\begin{itemize}
		\item $f_n$ is zero on the intervals $ [0,\frac{1}{2}-\frac{1}{n+2}] $ and $ [\frac{1}{2}+\frac{1}{n+2},1] $,
		\item $f_n(\frac{1}{2})=1$,
		\item $f_n$ is linear on the intervals $ [\frac{1}{2}-\frac{1}{n+2},\frac{1}{2}] $ and $ [\frac{1}{2},\frac{1}{2}+\frac{1}{n+2}] $.
	\end{itemize}
	Then $f_n \downarrow 0$ in $C[0,1]$ but $f_n \downarrow \chi_{\frac{1}{2}}$ in $B([0,1])$. So by Lemma 2.5 in \cite{GTX}, we have that  $Y$ is not regular in $X$. 	
\end{example}

Consider the sublattice $c_0$ of $\ell_\infty$. Then $(c_0,\lVert \cdot \rVert_\infty, {\mathbb R})$ is not $up$-regular in the LNVL $(\ell_\infty,\lVert \cdot \rVert_\infty, {\mathbb R})$. Indeed, $(e_n)$ is $un$-convergent in $c_0$ but not in $\ell_\infty$. However, $(c_0,\lvert \cdot \rvert,\ell_\infty)$ is $up$-regular in the LNVL $(\ell_\infty,\lvert \cdot\rvert,\ell_\infty)$.

\subsection{Several basic results}
We begin with the following result which is a $p$-version of \cite[Thm.4.3]{KMT}

\begin{theorem}\label{$up$-regular}
	Let $Y$ be a sublattice of an LNVL $X=(X,p,E)$. Then $Y$ is $up$-regular in each of the following cases:
	
	$(i)$\ $Y$ is majorizing in $X$;
	
	$(ii)$\ $Y$ is $p$-dense in $X$;
	
	$(iii)$\ $Y$ is a projection band in $X$.
\end{theorem}

\begin{proof}
	Let $(y_\alpha)\subseteq Y$ be such that $y_\alpha\xrightarrow[]{up} 0$ in $Y$. Let $0\neq x\in X_+$.
	
	$(i)$ There exists $y\in Y$ such that $x\leq y$. It follows from
	$$
	0\leq |y_\alpha|\wedge x\leq |y_\alpha|\wedge y\xrightarrow{p}0,
	$$
	that $y_\alpha\xrightarrow{up}0$ in $X$.
	
	$(ii)$ Choose an arbitrary $0\ne u\in p(X)$. Then there exists $y\in Y$ with $p(x-y)\le u$. Since
	$$
	|y_\alpha|\wedge x\le |y_\alpha|\wedge |x-y|+|y_\alpha|\wedge |y|,
	$$
	then 
	$$
	p(|y_\alpha|\wedge x)\le p(|y_\alpha|\wedge |x-y|)+p(|y_\alpha|\wedge |y|)\le u+p(|y_\alpha|\wedge |y|).
	$$
	Since $0\ne u\in p(X)$ is arbitrary and $|y_\alpha|\wedge |y|\xrightarrow{p}0$, then $|y_\alpha|\wedge x\xrightarrow{p}0$.
	Hence $y_\alpha\xrightarrow{up}0$ in $X$.
	
	$(iii)$ $Y=Y^{\bot\bot}$ implies that $X=Y\oplus Y^{\bot}$. Hence $x=x_1+x_2$ with $x_1\in Y$ and $x_1\in Y^{\bot}$
	Since $y_\alpha\wedge x_2=0$, we have 
	$$
	p(y_\alpha\wedge x)=p(y_\alpha\wedge(x_1+x_2))=p(y_\alpha\wedge x_1)\xrightarrow{o}0.
	$$ 
	Hence $y_\alpha\xrightarrow{up}0$ in $X$.
\end{proof}

The following result deals with a particular case of Example \ref{ExLNVL_5}.

\begin{theorem}\label{ideal is up-regular}
	Let $X$ be a vector lattice and $Y=X_n^\sim$. Assume $X_n^\sim$ separates the points of $X$. Define $p:X\to\mathbb{R}^Y$ by $p(x)[y]=|y|(|x|)$. 
	Then any ideal of $X$ is $up$-regular in $(X,p,\mathbb{R}^Y)$.
\end{theorem}

\begin{proof}
	Let $I$ be an ideal of $X$ and $x_\alpha$ be a net in $I$ such that $x_\alpha\xrightarrow[]{up} 0$ in $I$. We show $x_\alpha\xrightarrow[]{up} 0$ in $X$. 
	By Example \ref{Zabeti}, this is equivalent to show $|x_\alpha|\wedge u\xrightarrow{|\sigma|(X,Y)} 0$ for any $u\in X_+$. 
	First note that if $v\in I^\perp$, then $|x_\alpha|\wedge |v|=0$, and so, for any $w\in (I\oplus I^\perp)_+$, we have $|x_\alpha|\wedge w\xrightarrow{|\sigma|(X,Y)} 0$. 
	Note also that $I\oplus I^\perp$ is order dense (see, e.g., \cite[Thm.3.3.(2)]{AB}). Let $u\in X_+$ and $y\in Y$, 
	then there is a net $w_\beta$ in $(I\oplus I^\perp)_+$ such that $w_\beta\uparrow u$, and so $|y|(w_\beta\wedge u)\uparrow |y|(u)$. Given $\varepsilon>0$. There is $\beta_0$ such that
	$$ 
	|y|(u)-|y|(w_{\beta_0}\wedge u)<\frac{\varepsilon}{2}.
	$$
	Also there is $\alpha_0$ such that 
	$$
	|y|(|x_\alpha|\wedge w_{\beta_0})<\frac{\varepsilon}{2}
	$$
	for all $\alpha\geq\alpha_0$. Taking into account the inequality $|a\wedge c -b\wedge c|\leq |a-c|$ (cf. \cite[Thm.1.6.(2)]{AB}) we have for any $\alpha\geq\alpha_0$,
	$$
	|y|(|x_\alpha|\wedge u)=|y|(|x_\alpha|\wedge u)-|y|(|x_\alpha|\wedge u\wedge w_{\beta_0}) +|y|(|x_\alpha|\wedge u\wedge w_{\beta_0}) 
	$$
	$$
	\leq |y|(u)-|y|(w_{\beta_0}\wedge u)+|y|(|x_\alpha|\wedge w_{\beta_0})<\varepsilon.
	$$
	Since $u\in X_+$ and $y\in Y$ are arbitrary, we get $|x_\alpha|\wedge u\xrightarrow{|\sigma|(X,Y)} 0$ for any $u\in X_+$, and this completes the proof.
\end{proof}

The next Corollary might be compared with \cite[Cor.4.6]{KMT}
\begin{corollary}
	Let $X$ be a vector lattice and $Y=X_n^\sim$ be the {\em order continuous dual}. Assume that $X_n^\sim$ separates the points of $X$. Define $p:X\to\mathbb{R}^Y$ by $p(x)[y]=|y|(|x|)$. 
	Then any sublattice of $X$ is $up$-regular in the LNVL $(X,p,\mathbb{R}^Y)$.
\end{corollary}

\begin{proof}
	Let $X_0$ be a sublattice of $X$ and $x_\alpha$ be a net in $X_0$ such that $x_\alpha\xrightarrow[]{up} 0$ in $X_0$. Let $I_{X_0}$ be the ideal generated by $X_0$ in $X$. 
	Then $X_0$ is majorizing in $I_{X_0}$ and, by Theorem \ref{$up$-regular}(i), we get $x_\alpha\xrightarrow[]{up} 0$ in $I_{X_0}$. Now, Theorem \ref{ideal is up-regular} implies that $x_\alpha\xrightarrow[]{up} 0$ in $X$.
\end{proof}

\subsection{Order completion}
Denote by $X^\delta$ the order completion of a vector lattice $X$.

\begin{lemma}
	Let $(X^\delta,p,E)$ be an LNVL, where $X^\delta$ is the order completion of $X$. For any sublattice $Y\subseteq X$, if $Y^\delta$ is $up$-regular in $X^\delta$, then $Y$ is $up$-regular in $X=(X,p|_X,E)$.   
\end{lemma}

\begin{proof}
	Take a net $(y_\alpha)_\alpha\subseteq Y$ such that $y_\alpha\xrightarrow{up}0$ in $Y$. 
	Then $p(|y_\alpha|\wedge u)\xrightarrow{o}0$ for all $u\in Y_+$. Let $w\in Y^\delta$ and, since $Y$ is majorizing in $Y^\delta$, 
	there exists $y\in Y$ such that $w\leq y$. Therefore, we obtain $y_\alpha\xrightarrow{up}0$ in $Y^\delta$. 
	Since $Y^\delta$ is $up$-regular in $X^\delta$, the net $y_\alpha$ is $up$-convergent to $0$ in $X^\delta$, 
	and so in $X$. 
\end{proof}

\begin{lemma}\label{$up$-regularinX}
	Let $(X^\delta,p,E) $ be an LNVL. For any sublattice $Y\subseteq X$, if $Y$ is $up$-regular in $X$, then $Y$ is $up$-regular in $X^\delta$. 
\end{lemma}

\begin{proof}
	Let $(y_\alpha)_\alpha\subseteq Y$ and $y_\alpha\xrightarrow{up}0$ in $Y$. By assumption $y_\alpha\xrightarrow{up}0$ in $X$. 
	Let $u\in X^\delta_+$, then there exists $x\in X$ such that $u\leq x$. Therefore, we obtain $p(|y_\alpha|\wedge u)\leq p(|y_\alpha|\wedge x)\xrightarrow{o}0$, 
	i.e. $y_\alpha\xrightarrow{up}0$ in $X^\delta$.
\end{proof}

In connection with Lemma \ref{$up$-regularinX}, the following question arises. 

\begin{problem}
	Is it true that $I^\delta$ is $up$-regular in $X^\delta$, whenever $I$ is a $up$-regular ideal in $X$?
\end{problem}

\begin{proposition}\label{ordercomplete$p$}
	Let $(X,p,E)$ be an LNVL. Define $p^{\delta}_L:X^{\delta}\rightarrow E^{\delta}$ and $p^{\delta}_U:X^{\delta}\rightarrow E^{\delta}$ as follows: 
	$p^{\delta}_L(z)=\sup\limits_{0\leq x\leq |z|}p(x)$ and $p^{\delta}_U(z)=\inf\limits_{|z|\leq x}p(x)$ for all $z\in X^\delta$ $($clearly, both $p^{\delta}_U$ 
	and $p^{\delta}_L$ are extensions of $p$$)$. Then$:$
	\begin{enumerate}
		\item[$\it{(i)}$]    $p^{\delta}_L$ is a monotone $E^\delta$-valued norm;
		\item[$\it{(ii)}$]   $p^{\delta}_U$ is a monotone $E^\delta$-valued seminorm;
		\item[$\it{(iii)}$]  if $X$ is $op$-continuous, then $p^{\delta}_U$ is $p$-continuous $($i.e. $z_\gamma\downarrow 0$ in $X^\delta$ implies $p^{\delta}_U(z_\gamma)\downarrow 0$ in $E^\delta$$)$;
		\item[$\it{(iv)}$]   if $X$ is $op$-continuous, then $p_U^\delta=p_L^\delta$. 
	\end{enumerate}
\end{proposition}

\begin{proof}
	$(i)$ Let $X^{\delta}\ni z\neq 0$. Since $X$ is order dense in $X^\delta$, there is $x\in X$ such that $0<x\leq|z|$, and so $p^{\delta}_L(z)\ge p(x)>0$. 
	
	Let $0\neq\lambda\in\mathbb{R}$, then
	$$
	p^{\delta}_L(\lambda z)=\sup\limits_{0\leq x\leq|\lambda z|}p(x)=\sup\limits_{0\leq\frac{1}{|\lambda|}x\leq |z|}p(x)=|\lambda|\sup\limits_{0\leq\frac{1}{|\lambda|}x\leq |z|}p({|\lambda|}^{-1}x)=|\lambda|p^\delta_L(z).
	$$
	
	Let $z,w\in X^\delta$, we show $p^{\delta}_L(z+w)\leq p^{\delta}_L(z)+p^{\delta}_L(w)$. Suppose $0\leq x\leq |z+w|$, then $0\leq x\leq |z|+|w|$. 
	By the Riesz Decomposition Property, there exist $x_1,x_2\in X$ such that $0\leq x_1\leq |z|$, $0\leq x_2\leq |w|$, and $x=x_1+x_2$. So 
	$$
	p(x)=p(x_1+x_2)\leq p(x_1)+p(x_2)\leq p^{\delta}_L(z)+p^{\delta}_L(w).
	$$ 
	Thus $p^{\delta}_L(z+w)=\sup\limits_{0\leq x\leq|z+w|}p(x)\leq p^{\delta}_L(z)+p^{\delta}_L(w)$.
	
	Now, we prove the monotonicity of the lattice norm $p_L^\delta$.
	If $|z|\leq|w|$ then, for any $x\in X$ with $0\leq x\leq|z|$, we get $0\leq x\leq|w|$. 
	So $\displaystyle\sup_{0\leq x\leq|z|}p(x)\leq\sup_{0\leq x\leq|w|}p(x)$ or $p_L^\delta(z)\leq p_L^\delta(w)$. 
	
	$(ii)$ We show firstly the triangle inequality. Let $z,w\in X^\delta$ and $x_1,x_2 \in X$ be such that $|z|\leq x_1$ and $|w|\leq x_2$, 
	then $|z+w|\leq|z|+|w|\leq x_1+x_2$. So 
	$$ 
	p^{\delta}_U(z+w)=\inf\limits_{|z+w|\leq x}p(x)\leq p(x_1+x_2)\leq p(x_1)+p(x_2).
	$$ 
	Thus $p^{\delta}_U(z+w)-p(x_1)\leq p(x_2)$ for any $x_2\in X$ with $|w|\leq x_2$. 
	Hence $p^{\delta}_U(z+w)-p(x_1)\leq p^{\delta}_U(w)$ or $p^{\delta}_U(z+w)-p^{\delta}_U(w)\leq p(x_1)$, 
	which holds for all $x_1\in X$ with $|z|\leq x_1$. 
	Therefore, $p^{\delta}_U(z+w)-p^{\delta}_U(w)\leq p^{\delta}_U(z)$ or $p^{\delta}_U(z+w)\leq p^{\delta}_U(w)+ p^{\delta}_U(z)$.
	
	Now, if $|z|\leq|w|$, then, for any $x\in X$ with $0<|w|\leq x$, we have $|z|\leq x$. So $\displaystyle\inf_{|w|\leq x}p(x)\geq\inf_{|z|\leq x}p(x)$ or $p_U^\delta(z)\leq p_U^\delta(w)$.  
	
	$(iii)$ Assume $z_\gamma\downarrow 0$ in $X^\delta$. Let $A=\{a\in X:z_\gamma\leq a\ \text{for}\ \text{some} \ \gamma\}$. Then $\inf A=0$. Indeed, if $0\leq x\leq a$ for all $a\in A$, 
	then $0\leq x\leq A_\gamma$ for all $\gamma$, where $A_\gamma=\{a\in X:z_\gamma\leq a\}$. So, by \cite[Lm.2.7]{GTX}, we have $x\leq z_\alpha$. Thus $x=0$.
	
	Clearly, $A$ is directed downward and dominates the net $(z_\alpha)_\alpha$. Since $X$ is $op$-continuous, then $p(A)\downarrow 0$ and, by the definition of $p_U^\delta$, 
	we get that $p(A)$ dominates the net $(p_U^\delta z_\alpha)$. Therefore, $p_U^\delta z_\alpha\downarrow 0$. 
	
	$(iv)$ Let $z\in X^\delta$, then $\displaystyle |z| = \sup \{x \in X: 0 \leq x \leq |z|\}$. By $(iii)$, we have
	$$
	p_U^\delta (z) = p_U^\delta (|z|) = \sup \{p_U^\delta(x): x \in X, 0 \leq x \leq |z|\} 
	$$
	$$
	=\sup \{p(x): x \in X, 0 \leq x \leq |z|\} =p_L^\delta (z).
	$$
\end{proof}

In connection with Proposition \ref{ordercomplete$p$}$(iv)$, the following question arises.

\begin{problem}
	Does the equality $p_U^\delta=p_L^\delta$ imply the $op$-continuity of $X$?
\end{problem}

\begin{proposition}
	Let $(X,p,E)$ be an LNVL. Then, for every net $x_\alpha$ in $X$,
	\begin{equation*}
	\begin{split}
	x_\alpha\xrightarrow{up}0 \ \ in \ (X,p,E) \ \ \Leftrightarrow \ \ x_\alpha\xrightarrow{up}0 \ \ in\  (X^\delta,p^{\delta},E^\delta),
	\end{split}
	\end{equation*}
	where $p^{\delta} = p^{\delta}_L$. 
\end{proposition}

\begin{proof}
	Assume $x_\alpha\xrightarrow{up}0$ in $(X,p,E)$. Then $p(|x_\alpha|\wedge x)\xrightarrow{o}0$ in $E$ for all $x\in X_+$, and so $p(|x_\alpha|\wedge x)\xrightarrow{o}0$ 
	in $E^\delta$ for all $x\in X_+$, by \cite[Cor.2.9]{GTX}.  Hence 	
	\begin{equation}\label{eqn1}
	p^\delta(|x_\alpha|\wedge x)\xrightarrow{o}0 
	\end{equation} 
	in $E^{\delta}$ for all $x\in X_+$. Let $u\in X^\delta_+$, then there exists $x_u\in X_+$ such that $u\leq x_u$, since $X$ majorizes $X^\delta$. 
	From (\ref{eqn1}) it follows that $p^\delta(|x_\alpha|\wedge u)\xrightarrow{o}0$ in $E^\delta$. Since $u\in X^\delta_+$ is arbitrary, then 
	$x_\alpha\xrightarrow{up}0$ in $(X^\delta,p^\delta,E^\delta)$.
	
	Conversely, assume $x_\alpha\xrightarrow{up}0$ in $(X^\delta,p^\delta,E^\delta)$ then, for all $u\in X^\delta_+$, $p^\delta(|x_\alpha|\wedge u)\xrightarrow{o}0$ in $E^\delta$. 
	In particular, for all $x\in X_+$, $p(|x_\alpha|\wedge x)=p^\delta(|x_\alpha|\wedge x)\xrightarrow{o}0$ in $E^\delta$. 
	By \cite[Cor.2.9]{GTX}, $p(|x_\alpha|\wedge x)\xrightarrow{o}0$ in $E$ for all $x\in X_+$. Hence $x_\alpha\xrightarrow{up}0$ in $(X,p,E)$. 
\end{proof}

\section{Mixed-normed spaces}

In this section, we study LNVLs with mixed lattice norms.

\subsection{Mixed norms}
Let $(X,p,E)$ be an LNS and $(E,\lVert\cdot\rVert)$ be a normed lattice. The \textit{mixed norm} on $X$ is defined by 
$$
p\text{-}\lVert x\rVert=\lVert p(x)\rVert \ \ \ (\forall x\in X).
$$ 
In this case the normed space $(X,p\text{-}\lVert\cdot\rVert)$ is called a \textit{mixed-normed space} (see, for example \cite[7.1.1, p.292]{K})

The next proposition follows directly from the basic definitions and results, so its proof is omitted.

\begin{proposition}\label{mixed}
	Let $(X,p,E)$ be an LNVL, $(E,\lVert\cdot\rVert)$ be a Banach lattice, and $(X,p\text{-}\lVert\cdot\rVert)$ be a mixed-normed space. The following statements hold:
	\begin{enumerate} 
		\item[(i)] if $(X,p,E)$ is $op$-continuous and $E$ is order continuous, then $(X,p\text{-}\lVert \cdot\rVert)$ is an order continuous normed lattice$;$
		\item[(ii)] if a subset $Y$ of $X$ is $p$-bounded (respectively, $p$-dense) in $(X,p,E)$, then $Y$ is norm bounded (respectively, norm dense) in  $(X,p\text{-}\lVert\cdot\rVert)$$;$
		\item[(iii)] if $e\in X$ is a $p$-unit and $E$ is order continuous, then $e$ is a quasi-interior point of  $(X,p\text{-}\lVert \cdot\rVert)$$;$
		\item[(iv)] if $(X,p,E)$ is a $p$-Fatou space and $E$ is order continuous, then  $p\text{-}\lVert \cdot\rVert$ is a Fatou norm, \cite[p.42]{LZ1}$;$
		\item[(v)] if $Y$ is a $p$-almost order bounded subset of $X$, then $Y$ is almost order bounded set in $(X,p\text{-}\lVert \cdot\rVert)$.
	\end{enumerate}
\end{proposition}

\begin{theorem}\label{pKB-pKB}
	Let $(X,p,E)$ and $(E,m,F)$ be two $p$-$KB$-spaces. Then the LNVL $(X,m\circ p,F)$ is also a $p$-KB-space.
\end{theorem}

\begin{proof}
	Let $0\leq x_\alpha\uparrow$ and $m\big(p(x_\alpha)\big)\le g\in F$. Since $0\leq p(x_\alpha)\uparrow<\infty$
	and since $(E,m,F)$ is a $p$-KB-space, then there exists $y\in E$ such that $m\big(p(x_\alpha)-y)\big)\to 0$. 
	Hence $p(x_\alpha)\uparrow y$. Thus the net $x_\alpha$ is increasing and $p$-bounded. 
	Since $X$ is $p$-$KB$-space, then there exists $x\in X$ such that $p(x_\alpha-x)\to 0$. 
	As $(X,m\circ p,F)$ is clearly $po$-continuous, then $m\big(p(x_\alpha-x)\big)\xrightarrow{o}0$ i.e. 
	$m\circ p(x_\alpha-x)\xrightarrow{o}0$. Thus $(X,m\circ p,F)$ is a $p$-KB-space.
\end{proof}

\begin{corollary}\label{pKB-KB}
	Let $(X,p,E)$ be a $p$-$KB$-space and $(E,\lVert\cdot\rVert)$ be a KB-space. Then $(X,p\text{-}\lVert\cdot\rVert)$ is a KB-space.
\end{corollary}

\subsection{$up$-Completeness}
The following well-known technical lemma is a particular case of Lemma \ref{ptechlemma2}. 

\begin{lemma}\label{techlemma2}
	Given a Banach lattice $(X,\left\|\cdot\right\|)$. If $x_\alpha\xrightarrow{\left\|\cdot\right\|}x$  and $x_\alpha$ is $o$-Cauchy, then $x_\alpha\xrightarrow{o}x$.
\end{lemma} 

Recall that a Banach lattice is called $un$-complete if every $un$-Cauchy net is $un$-convergent, \cite{KMT}.

\begin{theorem}\label{up-complete}
	Let $(X,p,E)$ be an LNVL and $(E,\left\|\cdot\right\|)$ be an order continuous Banach lattice. 
	If $(X,p\text{-}\Vert\cdot\rVert)$ is a $un$-complete Banach lattice, then $X$ is $up$-complete. 
\end{theorem}

\begin{proof}
	Let $x_\alpha$ be a $up$-Cauchy net in $X$. So, for every $u\in X_+$, $p(|x_\alpha-x_\beta|\wedge u)\xrightarrow{o}0$. 
	Since $E$ is order continuous, then, for every $u\in X_+$, $\lVert p(|x_\alpha-x_\beta|\wedge u)\rVert\to 0$ 
	or for every $u\in X_+$, $p\text{-}\big\lVert |x_\alpha-x_\beta|\wedge u\big\rVert\to 0$, i.e. $x_\alpha$ is $un$-Cauchy in $(X, p\text{-}\lVert \cdot\rVert)$. 
	Since $(X,p\text{-}\lVert \cdot\rVert)$ is $un$-complete, then there exists $x\in X$ such that $x_\alpha\xrightarrow{un}x$ in $(X, p\text{-}\lVert \cdot\rVert)$. 
	That is, for every $u\in X_+$, $\left\|p(|x_\alpha-x|\wedge u)\right\|\to 0$. Next we show the net $\big(p(\lvert x_\alpha-x\rvert\wedge u)\big)_\alpha$ is order Cauchy in $E$. Indeed,
	$$\big\lvert p(\lvert x_\alpha-x\rvert\wedge u) - p(\lvert x_\beta-x\rvert\wedge u)\big\rvert \leq p\big(\big\lvert \lvert x_\alpha-x\rvert\wedge u - \lvert x_\beta-x\rvert\wedge u \big\rvert\big)  \leq p(\lvert x_\alpha-x_\beta\rvert\wedge u)\xrightarrow{o}0.$$ Now, Lemma \ref{techlemma2} above, implies that $p(\lvert x_\alpha-x\rvert\wedge u)\xrightarrow{o}0$.
\end{proof}

\subsection{$up$-Null nets and $up$-null sequences in mixed-normed spaces}
The following theorem is a $p$-version of \cite[Thm.3.2]{DOT} and a generalization of \cite[Lm.6.7]{GTX}, as we take $(X,p,E)=(X,\left\|\cdot\right\|,\mathbb R)$.

\begin{theorem}\label{$p$-version of Theorem 3.2 in DOT}
	Let $(X,p,E)$ be an $op$-continuous and $p$-complete LNVL, $E$ an order continuous Banach lattice, and $X\ni x_\alpha\xrightarrow{up}0$.
	Then there exist an increasing sequence $\alpha_k$ of indices and a disjoint sequence $d_k\in X$ such that $(x_{\alpha_k}-d_k)\xrightarrow{p}0$ as $k\to\infty$.
\end{theorem}

\begin{proof}
	Consider the mixed norm $p\text{-}\lVert x \rVert=\lVert p(x)\rVert$. Since $p(|x_\alpha|\wedge u)\xrightarrow{o}0$ for all $u\in X_+$, 
	then $p\text{-}\lVert|x_\alpha|\wedge u\rVert=\lVert p(|x_\alpha|\wedge u)\rVert \xrightarrow{o}0$ that means 
	$x_\alpha \xrightarrow{un}0$ in $(X,p\text{-}\lVert\cdot\rVert)$ by  $o$-continuity of $(E,\lVert\cdot\rVert)$.
	
	By \cite[Thm.3.2]{DOT}, there exists an increasing sequence $\alpha_n$ of indices and a disjoint sequence $d_n$ in $X$ such that 
	$ x_{\alpha_n}-d_n \xrightarrow{p\text{-}\lVert\cdot\rVert} 0$. Now $(X,p\text{-}\lVert\cdot\rVert)$ is a Banach lattice by \cite[7.1.3 (1), p.294]{K}. 
	So, by \cite[Thm.VII.2.1]{V} there is a further subsequence $(\alpha_{n_k})$ such that $\lvert x_{\alpha_{n_k}}-d_{n_k}\rvert\xrightarrow{o}0$ in $X$.  
	By $op$-continuity of $X$, $p(x_{\alpha_{n_k}}-d_{n_k})\xrightarrow{o}0$. 
\end{proof}

The next corollary is a $p$-version of \cite[Cor.3.5]{DOT}.

\begin{corollary}
	Let $(X,p,E)$ be an $op$-continuous LNVL, $E$ be an order continuous Banach lattice, and $X\ni x_\alpha\xrightarrow{up}0$.
	Then there exist an increasing sequence $\alpha_k$ of indices such that $x_{\alpha_k}\xrightarrow{up}0$.
\end{corollary}

\begin{proof}
	Let $\alpha_k$ and $d_k$ be as in Theorem \ref{$p$-version of Theorem 3.2 in DOT}. 
	Since the sequence $d_k$ is disjoint, then $d_k\xrightarrow{uo}0$ by \cite[Cor.3.6.]{GTX}. 
	Since $X$ is $op$-continuous, then $d_k\xrightarrow{up}0$. Since
	$$
	p(|x_{\alpha_k}-d_k|\wedge u)\le p(x_{\alpha_k}-d_k)\xrightarrow{o}0 \ \ \ (\forall u\in X_+),
	$$
	then $x_{\alpha_k}-d_k\xrightarrow{up}0$. Since $d_k\xrightarrow{up}0$, then $x_{\alpha_k}\xrightarrow{up}0$.
\end{proof}

Next proposition extends \cite[Prop.4.1]{DOT} to LNVLs.

\begin{proposition}\label{16}
	Let $(X,p,E)$ be a $p$-complete LNVL, $(E,\lVert\cdot \rVert)$ be an order continuous Banach lattice, and $X\ni x_n\xrightarrow{up}0$. 
	Then there exist a subsequence $x_{n_k}$ of $x_n$ such that $x_{n_k}\xrightarrow{uo}0$ as $k\to\infty$.
\end{proposition}

\begin{proof}
	Suppose $x_n\xrightarrow{up}0$, then, for all $u\in X_+$ $p(|x_n|\wedge u)\xrightarrow{o}0$, and so $\|p(|x_n|\wedge u)\|\to 0$ since $E$ is order continuous. 
	Thus $|x_n|\wedge u\xrightarrow{p\text{-}\lVert\cdot\rVert}0$, i.e. $x_n\xrightarrow{un}0$ in $(X,p\text{-}\lVert\cdot\rVert)$. It follows from \cite[7.1.2, p.293]{K} that the mixed-normed space 
	$(X,p\text{-}\lVert\cdot\rVert)$ is a Banach lattice, and so by \cite[Prop.4.1]{DOT} there is a subsequence $x_{n_k}$ of $x_n$ such that $x_{n_k}\xrightarrow{uo}0$ as $k\to\infty$.
\end{proof}

Next result is a $p$-version of \cite[Thm.4.4]{DOT}.

\begin{proposition}
	Let $(X,p,E)$ be an $op$-continuous and $p$-complete LNVL such that $(E,\lVert \cdot \rVert)$ is an order continuous atomic Banach lattice. 
	Then a sequence in $X$ is $up$-null iff every subsequence has a further subsequence which $uo$-converges to zero.
\end{proposition}

\begin{proof}
	The forward implication follows from Proposition \ref{16}. Conversely, let $x_n$ be a sequence in $X$ and assume that 
	$x_n \not \xrightarrow {up}0$. Then there is an atom $a\in E_+$, $u \in X_+$, $\varepsilon_0 >0$ and a subsequence $x_{n_k}$ of $x_n$ satisfying 
	$f_a\big(p(|x_{n_k}| \wedge u)\big) \geq \varepsilon_0$ for all $k$. By the hypothesis there exist a further subsequence $x_{n_{k_j}}$ of 
	$x_{n_k}$ which $uo$-converges to zero. By the $op$-continuity of $X$ we get $p(|x_{n_{k_j}}| \wedge u) \xrightarrow {o}0$, 
	and so $f_a\big(p(|x_{n_{k_j}}| \wedge u)\big) \rightarrow 0$, which is a contradiction.
\end{proof}

Our last result is a $p$-version of \cite[Lm.5.1]{DOT}.

\begin{proposition}
	Let $(X,p,E)$ be an $op$-continuous $p$-complete LNVL and $(E,\lVert\cdot\rVert)$ be an order continuous Banach lattice. If $X$ is atomic and $x_n$ is an order bounded sequence such that $x_n\xrightarrow{p}0$ in $X$, 
	then $x_n\xrightarrow{o}0$.
\end{proposition}

\begin{proof}
	The mixed-normed space $(X,p\text{-}\lVert\cdot\rVert)$ is an atomic order continuous Banach lattice such that $x_n\xrightarrow{p\text{-}\lVert\cdot\rVert}0$, and so $x_n\xrightarrow{o} 0$ by \cite[Lm.5.1]{DOT}.
\end{proof}

\bibliographystyle{amsplain}

\begin{thebibliography}{99}
\bibitem{AA}
Y. A. Abramovich and C. D. Aliprantis, 
\textit{An Invitation to Operator Theory}, 
Graduate Studies in Mathematics,
Vol.50, American Mathematical Society, Providence, RI,
2002.

\bibitem{AB}
C. D. Aliprantis and O. Burkinshaw,
\textit{Positive Operators}, 
Academic Press, Orlando, London, 1985.

\bibitem{AEEM2}
A. Ayd{\i}n, E. Y. Emelyanov, N. Erkur\c{s}un \"Ozcan, and M. A. A. Marabeh, 
\textit{Compact-Like Operators in Lattice-Normed Spaces},
arXiv:1701.03073v2.

\bibitem{DEM-a}
Y. A. Dabboorasad, E. Y. Emelyanov, and M. A. A. Marabeh,  
\textit{$um$-Topology in multi-normed vector lattices},
Positivity (2017). https://doi.org/10.1007/s11117-017-0533-6

\bibitem{DEM-b}
Y. A. Dabboorasad, E. Y. Emelyanov, and M. A. A. Marabeh,  
\textit{$u\tau$-Convergence in locally solid vector lattices}, arXiv:1706.02006v3.
 
\bibitem{D}
R. DeMarr, 
\textit{Partially ordered linear spaces and locally convex linear topological spaces}, 
Illinois J. Math. {\bf 8} (1964), no. 4, 601--606. 

\bibitem{DOT}
Y. Deng, M. O'Brien, and V. G. Troitsky, 
\textit{Unbounded norm convergence in Banach lattices}, 
Positivity, {\bf 21} (2017), no. 3, 963--974.

\bibitem{EEG}
E. Y. Emelyanov, N. Erkur\c{s}un \"Ozcan, and S. G. Gorokhova, 
\textit{Koml\'os Properties in Banach Lattices}
arXiv:1710.02580v1.

\bibitem{EM}
E. Y. Emelyanov and M. A. A. Marabeh, 
\textit{Two measure-free versions of the Brezis-Lieb lemma},
Vladikavkaz. Mat. Zh. {\bf 18} (2016), no.1, 21--25.

\bibitem{E}   	
E. Y. Emelyanov, 
\textit{Infinitesimal analysis and vector lattices}, 
Siberian Adv. Math. {\bf 6} (1996), no.1, 19--70.

\bibitem{EV}
Z. Ercan and M. Vural, 
\textit{Towards a theory of unbounded locally solid Riesz spaces},
arXiv:1708.05288v1.

\bibitem{G}
N. Gao, 
\textit{Unbounded order convergence in dual spaces}, 
J. Math. Anal. Appl. {\bf 419} (2014), no. 1, 347--354.

\bibitem{GLX}
N. Gao, D. H. Leung, and F. Xanthos, 
\textit{Duality for unbounded order convergence and applications}, 
Positivity (2017). https://doi.org/10.1007/s11117-017-0539-0

\bibitem{GTX}
N. Gao, V. G. Troitsky, and F. Xanthos, 
\textit{Uo-convergence and its applications to Ces\'aro means in Banach lattices}, 
Israel J. Math. {\bf 220} (2017), no. 2, 649--689. 

\bibitem{GX}
N. Gao and F. Xanthos, 
\textit{Unbounded order convergence and application to martingales without probability}, 
J. Math. Anal. Appl. {\bf 415} (2014), no. 2, 931--947.

\bibitem{KMT}
M. Kandi\'c, M. A. A. Marabeh, and V. G. Troitsky, 
\textit{Unbounded norm topology in Banach lattices},
J. Math. Anal. Appl. {\bf 451} (2017), no. 1, 259--279.

\bibitem{KT}
M. Kandi\'c and M. A. Taylor, 
\textit{Metrizability of minimal and unbounded topologies},
arXiv:1709.05407v1.

\bibitem{Ka}
S. Kaplan, 
\textit{On unbounded order convergence},
R. Anal. Exchange {\bf 23} (1997/98), no. 1, 175--184.

\bibitem{KK}
A. G. Kusraev and S. S. Kutateladze, 
\textit{Boolean Valued Analysis}, 
Mathematics and its Applications, Springer Science+Business Media, Dordrecht, 1999.

\bibitem{K}
A. G. Kusraev, 
\textit{Dominated Operators}, 
Mathematics and its Applications, Springer Science+Business Media, Dordrecht, 2000.

\bibitem{LZ1}
W. A. J. Luxemburg and A. C. Zaanen, 
\textit{Riesz Spaces I}, 
North-Holland Publishing Co., Amsterdam, 1971.

\bibitem{M}
M. Marabeh, 
\textit{Brezis-Lieb lemma in convergence vector lattices},
preprint.

\bibitem{N}
H. Nakano, 
\textit{Ergodic theorems in semi-ordered linear spaces}, 
Ann. of Math. (2) {\bf 49} (1948), no. 3, 538--556.

\bibitem{Tay1}
M. A. Taylor, 
\textit{Unbounded topologies and uo-convergence in locally solid vector lattices}, 
arXiv:1706.01575v1.

\bibitem{Tay2}
M. A. Taylor, 
\textit{Completeness of Unbounded Convergences},
arXiv:1708.06885v1.

\bibitem{T}
V. G. Troitsky, 
\textit{Measures of non-compactness of operators on Banach lattices}, 
Positivity {\bf 8} (2004), no.2, 165--178.

\bibitem{V}
B. Z. Vulikh, 
\textit{ Introduction to the Theory of Partially Ordered Spaces}, 
Wolters-Noordhoff Scientific Publications, Ltd., Groningen, 1967.

\bibitem{W}
A. W. Wickstead, 
\textit{Weak and unbounded order convergence in Banach lattices}, 
J. Austral. Math. Soc. Ser. A, {\bf 24} (1977), no.3, 312--319.

\bibitem{Za}
A. C. Zaanen, 
\textit{Riesz spaces II}, 
North-Holland Publishing Co., Amsterdam, 1983.

\bibitem{Z}
O. Zabeti,
\textit{ Unbounded absolute weak convergence in Banach lattices}, 
Positivity (2017). https://doi.org/10.1007/s11117-017-0524-7


\end{thebibliography}

\end{document}